\documentclass[12pt,a4paper,leqno]{amsart}

\usepackage[latin1]{inputenc}
\usepackage[T1]{fontenc}
\usepackage{amsfonts}
\usepackage{amsmath}
\usepackage{amssymb}
\usepackage{eurosym}
\usepackage{mathrsfs}
\usepackage{palatino}
\usepackage{color}
\usepackage{esint}
\usepackage{url}

\newcommand{\R}{\mathbb{R}}

\newcommand{\N}{\mathbb{N}}

\newcommand{\Z}{\mathbb{Z}}

\numberwithin{equation}{section}
\usepackage{graphicx}
\newcommand{\osc}[1]{\omega_{#1}}

\newcommand{\dist}[0]{\operatorname{dist}}

\newcommand{\eps}[0]{\varepsilon}

\swapnumbers
\theoremstyle{plain}
\newtheorem{thm}[equation]{Theorem}
\newtheorem{lem}[equation]{Lemma}

\newtheorem{cor}[equation]{Corollary}

\theoremstyle{definition}

\theoremstyle{remark}
\newtheorem{rem}[equation]{Remark}

\title{Square functions with general measures II}

\author{Henri Martikainen}
\address{D\'epartement de Math\'ematiques, B\^atiment 425, Facult\'e des Sciences d'Orsay, Universit\'e Paris-Sud 11, F-91405 Orsay Cedex}
\email{henri.martikainen@math.u-psud.fr}
\thanks{H.M. is supported by the Emil Aaltonen Foundation. M.M. is supported by Fondation de Math\'ematiques Jacques Hadamard (FMJH).
The research of T.O. is partially supported by the Academy of Finland, grant 133264. The two first named authors wish to thank Universit\'e Paris-Sud 11, Orsay, for its hospitality.}

\author{Mihalis Mourgoglou}
\address{D\'epartement de Math\'ematiques, B\^atiment 425, Facult\'e des Sciences d'Orsay, Universit\'e Paris-Sud 11, F-91405 Orsay Cedex}
\email{mihalis.mourgoglou@math.u-psud.fr}

\author{Tuomas Orponen}
\address{Department of Mathematics and Statistics, University of Helsinki, P.O.B. 68, FI-00014 Helsinki, Finland}
\email{tuomas.orponen@helsinki.fi}

\makeatletter
\@namedef{subjclassname@2010}{%
  \textup{2010} Mathematics Subject Classification}
\makeatother


\begin{document}

\begin{abstract}
We continue developing the theory of conical and vertical square functions on $\R^{n}$, where $\mu$ is a power bounded measure, possibly non-doubling. We provide new boundedness criteria and construct various counterexamples.

First, we prove a general local $Tb$ theorem with tent space $T^{2,\infty}$ type testing conditions to characterise the $L^{2}$ boundedness. Second, we completely answer the question, whether the boundedness of our operators on $L^{2}$ implies boundedness on other $L^{p}$ spaces, including the endpoints. For the conical square function, the answers are generally affirmative, but the vertical square function can be unbounded on $L^{p}$ for $p > 2$, even if $\mu = dx$. For this, we present a counterexample.  Our kernels $s_t$, $t > 0$, do not necessarily satisfy any continuity in the first variable -- a point of technical importance throughout the paper.

Third, we construct a non-doubling Cantor-type measure and an associated conical square function operator, whose $L^{2}$ boundedness depends on the exact aperture of the cone used in the definition. Thus, in the non-homogeneous world, the 'change of aperture' technique -- widely used in classical tent space literature -- is not available. Fourth, we establish the sharp $A_{p}$-weighted bound for the conical square function under the assumption that $\mu$ is doubling.
\end{abstract}

\maketitle

\tableofcontents

\section{Introduction}
This paper is concerned with the $L^{p}$ theory of the conical and vertical square functions operators $S$ and $V$, defined formally on complex valued functions $f$ on $\R^{n}$ by
\begin{displaymath} Sf(x) = \Big( \iint_{\Gamma(x)} |\theta_t f(y)|^2 \,\frac{d\mu(y) dt}{t^{m+1}}\Big)^{1/2} \quad \text{and} \quad Vf(x) = \left( \int_{0}^{\infty} |\theta_{t}f(y)|^{2} \, \frac{dt}{t} \right)^{1/2}. \end{displaymath}
Here $\Gamma(x)$, $x \in \R^{n}$, is the cone $\Gamma(x) = \{(y,t) \in \R^{n+1}_+\colon \, |x-y| < t\}$, and $\theta_{t}$, $t > 0$, is a linear operator to be specified momentarily. A major part of the paper is couched in the non-homogeneous setting, meaning precisely that the Borel measure $\mu$ and the exponent $m > 0$ above are related by the condition
\begin{equation}\label{powerBound} \mu(B(x,r)) \lesssim r^{m}, \qquad x \in \R^{n}, \; r > 0. \end{equation}
The only exception to this rule is the final chapter, where weighted theory for $S$ is established under the extra assumption that the measure $\mu$ is doubling.

The linear operators $\theta_{t}$, $t > 0$, have the form
\begin{displaymath} \theta_{t}f(x) = \int_{\R^{n}} s_{t}(x,y)f(y) \, d\mu(y), \end{displaymath}
where, for some fixed exponent $\alpha > 0$, the kernels $s_{t}$ satisfy the size and continuity conditions
\begin{equation}\label{eq:size}
|s_t(x,y)| \lesssim \frac{t^{\alpha}}{(t+|x-y|)^{m+\alpha}}
\end{equation}
and
\begin{equation}\label{eq:hol}
|s_t(x,y) - s_t(x,z)| \lesssim \frac{|y-z|^{\alpha}}{(t+|x-y|)^{m+\alpha}} \end{equation}
whenever $|y-z| < t/2$. It is worth emphasising that no regularity from $s_{t}$ is required in the first variable.

A prequel to this article is \cite{MM} by the first two authors, establishing a $Tb$ type theorem to characterise the $L^{2}(\mu)$-boundedness of $S$ and $V$. Namely, the operator $V$ is bounded on $L^{2}(\mu)$, if and only if there exists a function $b \in L^{\infty}(\mu)$ such that
\begin{displaymath}
\Big| \int_Q b(x) \,d\mu(x) \Big| \gtrsim \mu(Q)
\end{displaymath}
and
\begin{displaymath}
\iint_{\widehat Q} |\theta_t b(x)|^2 \,d \mu(x) \frac{dt}{t} \lesssim \mu(3Q)
\end{displaymath}
for every cube $Q \subset \R^n$. Here $\widehat{Q}$ is the Carleson box $\widehat Q = Q \times (0, \ell(Q)) \subset \R^{n} \times \R_{+}$. 

The $L^{2}(\mu)$ boundedness of the operator $S$ reduces to this by observing that $\|Sf\|_{L^{2}(\mu)} = \|\tilde{V}f\|_{L^{2}(\mu)}$, where $\tilde{V}$ is the vertical square function with kernel
\begin{displaymath} \tilde{s}_{t}(x,y) = \left(\frac{\mu(B(x,t))}{t^{m}}\right)^{1/2}s_{t}(x,y). \end{displaymath}
Since the $x$-continuity of the kernel $\tilde{s}_{t}$ is required neither here nor in \cite{MM}, the kernel $\tilde{s}_{t}(x,y)$ still satisfies the assumptions \eqref{eq:size} and \eqref{eq:hol}. Consequently, the $Tb$ theorem in \cite{MM} applies directly to $\tilde{V}$, characterising the $L^{2}(\mu)$-boundedness of $S$.

The first contribution of the present paper is to prove a \textbf{local} $Tb$ theorem for the operators $V$ and $S$. The result is, again, only formulated for $V$, but the reduction above shows how it can also be applied to $S$. 
\begin{thm}\label{localTb} Assume that to every cube $Q \subset \R^n$ there is associated a function $b_Q$ which satisfies:
\begin{enumerate}
\item spt$\,b_Q \subset Q$;
\item $|\langle b_Q \rangle_Q| \gtrsim 1$;
\item $\|b_Q\|_{L^{\infty}(\mu)} \lesssim 1$;
\item $\sup_{R \subset \R^n \textup{ cube}} \mu(3R)^{-1} \iint_{\widehat R} |\theta_t b_Q(x)|^2 d\mu(x) dt/t \lesssim 1$.
\end{enumerate}
Then $V$ is bounded on $L^{2}(\mu)$. 
\end{thm}
The assumptions are scale invariant (of type $L^{\infty}$ and $T^{2,\infty}$), which seems to be the best one can currently do with general measures.
This corresponds with the most general assumptions in the Calder\'on--Zygmund case by Nazarov--Treil--Volberg \cite{NTVa}. However, they are only able to use
BMO type bounds for $Tb_Q$ assuming that the kernel of $T$ is antisymmetric (with $L^{\infty}$ bounds for $Tb_Q$ the kernel can be general). Here,
in the square function world, one can cope with (4) without posing any such further restrictions.

In the Calder\'on--Zygmund world the usage of probabilistic techniques in connection with \textbf{local} $Tb$ theorems is a surprisingly delicate matter for various reasons.
This has been elaborated in Remark 4.1 \cite{HM} and in \cite{LV} (see especially Remark 2.14). We note that in our situation 
the proof is based on the averaging identity over good Whitney regions proved in \cite{MM}. This completely avoids all the technicalities here.

After the proof of Theorem \ref{localTb}, the paper studies the boundedness of $S$ and $V$ on some other spaces: assuming that $S$ (resp. $V$) is bounded on $L^{2}(\mu)$, does it follow that $S$ (resp. $V$) is then also bounded on other Lebesgue spaces, including the endpoints? The answer for $V$ is more entertaining because of the negative answer in $L^p(\mu)$ for $p > 2$ -- for this we construct a counterexample.
In any case, we formulate a complete answer to this question in the next theorem.
\begin{thm}\label{LpSF} 
\begin{itemize}
\item[(S)] Assume that $S$ is bounded on $L^{2}(\mu)$. Then $S$ is also a bounded mapping $L^{1}(\mu) \to L^{1,\infty}(\mu)$ and $L^{\infty}(\mu) \to \textup{RBMO}(\mu)$. Consequently, $S$ is bounded on $L^{p}(\mu)$ for all $1 < p <  \infty$.
\item[(V)] Assume that $V$ is bounded on $L^{2}(\mu)$. Then $V$ is also a bounded mapping $L^{1}(\mu) \to L^{1,\infty}(\mu)$. Consequently, $V$ is bounded on $L^{p}(\mu)$ for all $1 < p \leq 2$. The operator $V$ need not map $L^{p}(\mu) \to L^{p}(\mu)$ for any $p > 2$, even if $\mu = dx$.
\end{itemize}
\end{thm}
There are some previous results in the literature similar to Theorem \ref{LpSF}(S). First, for doubling measures $\mu$, the first part of our theorem follows from the work of Harboure, Torrea and Viviani, the accurate citation being
Theorem 4.4 of \cite{HTV}. The strategy in \cite{HTV} is to interpret $S$ as a vector-valued Calder\'on-Zygmund operator, after which the extension from $L^{2}$ to $L^{p}$ and the endpoints follows immediately from the well-developed theory for such operators. Curiously, in the world of doubling measures, such an interpretation can be made even without assuming the first variable continuity of the kernel $s_{t}$. And, indeed, no such assumption is made in \cite{HTV}. 

However, the words 'in the world of doubling measures' cannot be dispensed with. In fact, the lack of first variable continuity is compensated by the well-known observation that the aperture of the cone $\Gamma(x)$ does not crucially affect the $L^{2}$-norm of $Sf$. As a consequence, one may safely replace indicator of the cone $\Gamma(x)$ by a slightly larger smooth approximation, without losing the $L^{2}$-boundedness of the operator in the process. This trick is used extensively in the paper \cite{HTV}, including the proof of \cite[Theorem 4.4]{HTV}. 

In the non-homogeneous world, the trick is no longer available. In fact, we show that the change of aperture technique, used widely in connection with tent spaces, fails with general measures.
This is the content of the paper's third theorem. 
\begin{thm}\label{noTricks} There exist a Borel probability measure $\mu$ and a square function operator $S$ on $\R$ with the following properties. 
\begin{itemize}
\item[(i)] The measure $\mu$ and the kernel of the operator $S$ satisfy the assumptions \eqref{powerBound}, \eqref{eq:size} and \eqref{eq:hol} for some $0 < m < 1$. 
\item[(ii)] The operator $S$ is bounded on $L^{2}(\mu)$, but $S_{\alpha}(1) \notin L^{2}(\mu)$ for any $\alpha > 1$, where
\begin{displaymath} S_{\alpha}f(x) = \Big(\iint_{\Gamma_{\alpha}(x)} |\theta_{t}f(y)|^{2} \, \frac{d\mu(y)dt}{t^{m + 1}}\Big)^{1/2} \end{displaymath}
is the square function associated with the cones 
\begin{displaymath} \Gamma_{\alpha}(x) = \{(y,t) \in \R \times \R_{+} : |y - x| <  \alpha t\}. \end{displaymath}
\end{itemize}
\end{thm}
This example actually has the following interesting corollary that $S$ can be bounded in $L^2(\mu)$ even if $V$ is not.
\begin{cor}\label{noTricksCor} If $V$ is the vertical square function operator associated with the same kernel and measure as $S$ from Theorem \ref{noTricks}, then $V(1) \notin L^{2}(\mu)$. In particular, the boundedness of $S$ on $L^{2}(\mu)$ does not imply the boundedness of $V$ on $L^{2}(\mu)$. 
\end{cor}

In the end, the proof of Theorem \ref{LpSF}(S) is more straightforward than that of \cite[Theorem 4.4]{HTV}. The strategy here is simply to establish the valid endpoint results directly (using Tolsa's non-homogeneous Calder\'on-Zygmund decomposition \cite{To2} in the $L^{1} \to L^{1,\infty}$ end) and then apply a recent interpolation theorem for sublinear operators by H. Lin and D. Yang \cite{LY} to obtain boundedness on the full range of exponents. A similar approach was used recently in \cite{XZ} to obtain an analogue of Theorem \ref{LpSF} for certain Littlewood-Paley operators (with $x$-continuous kernels) in the non-homogeneous setting. 

We finish by making a note about the sharp $A_p$ theory for $S$ assuming that the measure $\mu$ is doubling. Even with doubling measures \textbf{sharp} weighted theory is essentially non-homogeneous analysis, since the doubling property of the weight cannot certainly be used. There exists closely related sharp weighted theory, but in the literature all of it is written for convolution type square functions.
Indeed, the sharp bound in the convolution case was first established by A. Lerner \cite{Le1}. This proof was heavily based on the usage of the intrinsic square function $G$
of J.M. Wilson \cite{Wi}. As such, the proof relies on the convolution structure. In a very recent work \cite{Le2} Lerner developed new techniques, which do not rely on the intrinsic square function $G$, in order to prove
sharp aperture-weighted estimates for square functions of convolution type. To complete our study of general square functions, we take this opportunity to record that these current techniques can also be used to obtain sharp weighted theory in our setting.
\begin{thm}\label{weighted}
If $\mu$ is doubling and $S$ is of weak-type $(1,1)$, then for any $p \in (1,\infty)$ there holds that
$$ \|Sf\|_{L^p(w)} \lesssim_{n,p} [w]_{A_p}^{\max{(\frac{1}{2},\frac{1}{p-1}})} \|f\|_{L^p(w)},$$
where $$[w]_{A_p}=\sup_Q \left(\fint_Q w \,d\mu \right) \left(\fint_Q w^{-\frac{1}{p-1}} \,d\mu \right)^{p-1} .$$
\end{thm}

\section{Local $Tb$ theorem}
In this section we prove Theorem \ref{localTb} -- that is, a general local $Tb$ theorem for our square functions. To this end,
assume that to every cube $Q \subset \R^n$ there is associated a function $b_Q$ which satisfies:
\begin{enumerate}
\item spt$\,b_Q \subset Q$;
\item $|\langle b_Q \rangle_Q| \gtrsim 1$;
\item $\|b_Q\|_{L^{\infty}(\mu)} \lesssim 1$;
\item $\sup_{R \subset \R^n \textup{ cube}} \mu(3R)^{-1} \iint_{\widehat R} |\theta_t b_Q(x)|^2 d\mu(x) dt/t \lesssim 1$.
\end{enumerate}
We will prove that this implies the square function bound
\begin{equation}\label{eq:SFbound}
\iint_{\R^{n+1}_+} |\theta_t f(x)|^2\,d\mu(x)\frac{dt}{t} \lesssim \|f\|_{L^2(\mu)}^2, \qquad f \in L^2(\mu).
\end{equation}

\subsection{Stopping times and the martingale difference operators $\Delta_Q$}
Let $\mathcal{D}$ be a dyadic system of cubes.
Let $s \in \N$ be an arbitrary large index and $Q_0 \in \mathcal{D}$ be a fixed cube with $\ell(Q_0) = 2^s$.
Let $\mathcal{D}^0 = \{Q_0\}$.

Let $\mathcal{D}^1 = \{Q^k_1\}_k$ consist of the maximal $\mathcal{D}$-cubes $Q \subset Q_0$ for which there holds
\begin{displaymath}
\Big| \int_Q b_{Q_0}\,d\mu \Big| < c\mu(Q).
\end{displaymath}
Here $c$ is a fixed small enough constant.
It follows that
\begin{displaymath}
\mu\Big( \bigcup_k Q^k_1 \Big) \le \tau \mu(Q_0)
\end{displaymath}
for some $\tau < 1$.

Next, fix a cube $Q^k_1$ and consider all the maximal $\mathcal{D}$-cubes $Q \subset Q^k_1$ for which there holds
\begin{displaymath}
\Big| \int_Q b_{Q^k_1}\,d\mu \Big| < c\mu(Q).
\end{displaymath}
We do this for every $Q^k_1 \in \mathcal{D}^1$, and call the resulting collection of cubes $\mathcal{D}^2 = \{Q^k_2\}_k$. We proceed to obtain collections
$\mathcal{D}^j$ for every $j$. For every $Q \in \mathcal{D}^j$ there holds
\begin{displaymath}
\mu\Big( \bigcup_{Q' \in \mathcal{D}^{j+1}, \, Q' \subset Q} Q' \Big) \le \tau\mu(Q).
\end{displaymath}

For every $Q \subset Q_0$ we let $Q^a$ be the smallest cube in the family $\bigcup \mathcal{D}^j$ containing $Q$. Note that if $Q \subset Q_0$ is such that
$Q^a \in \mathcal{D}^t$, there holds for every $j \ge 1$ that
\begin{displaymath}
\mu\Big( \bigcup_{Q' \in \mathcal{D}^{t+j}, \, Q' \subset Q} Q'\Big) = \sum_{Q' \in \mathcal{D}^{t+j}, \, Q' \subset Q} \mu(Q') \le \tau^{j-1}\mu(Q).
\end{displaymath}
The next lemma follows.
\begin{lem}\label{lem:mescar}
The following is a Carleson sequence: $\alpha_Q = 0$ if $Q$ is not from $\bigcup_j \mathcal{D}^j$, and it equals $\mu(Q)$ otherwise.
This means that $\sum_{Q \subset R} a_Q \lesssim \mu(R)$ for every dyadic $R$.
\end{lem}
Given a cube $Q$ let ch$(Q)$ consist of those cubes $Q' \subset Q$ for which $\ell(Q') = \ell(Q)/2$.
We define
\begin{equation}\label{eq:mdo}
\Delta_Q f = \sum_{Q' \in \, \textrm{ch}(Q)} \Big[\frac{\langle f \rangle_{Q'}}{\langle b_{(Q')^a}\rangle_{Q'}}b_{(Q')^a} - \frac{\langle f \rangle_Q}{\langle b_{Q^a}\rangle_Q}b_{Q^a}\Big]1_{Q'}.
\end{equation}

\subsection{Random dyadic grids}
At this point we need to insert a standard disclaimer about random dyadic grids (these facts are essentially presented in this way by Hyt\"onen in \cite{Hy}).
To this end, let us be given a random dyadic grid $\mathcal{D} = \mathcal{D}(w)$, $w = (w_i)_{i \in \Z} \in (\{0,1\}^n)^{\Z}$.
This means that $\mathcal{D} = \{Q + \sum_{i:\, 2^{-i} < \ell(Q)} 2^{-i}w_i: \, Q \in \mathcal{D}_0\} = \{Q + w: \, Q \in \mathcal{D}_0\}$, where we simply have defined
$Q + w := Q + \sum_{i:\, 2^{-i} < \ell(Q)} 2^{-i}w_i$. Here $\mathcal{D}_0$ is the standard dyadic grid of $\R^n$.

We set $\gamma = \alpha/(2m+2\alpha)$, where $\alpha > 0$ appears in the kernel estimates and $m$ appears in $\mu(B(x,r)) \lesssim r^m$.
A cube $Q \in \mathcal{D}$ is called bad if there exists another cube $\tilde Q \in \mathcal{D}$ so that $\ell(\tilde Q) \ge 2^r \ell(Q)$ and $d(Q, \partial \tilde Q) \le \ell(Q)^{\gamma}\ell(\tilde Q)^{1-\gamma}$.
Otherwise it is good.
One notes that $\pi_{\textrm{good}} := \mathbb{P}_{w}(Q + w \textrm{ is good})$ is independent of $Q \in \mathcal{D}_0$. The parameter $r$ is a fixed constant so large that $\pi_{\textrm{good}} > 0$
and $2^{r(1-\gamma)} \ge 3$.

Furthermore, it is important to note that for a fixed $Q \in \mathcal{D}_0$
the set $Q + w$ depends on $w_i$ with $2^{-i} < \ell(Q)$, while the goodness (or badness) of $Q + w$ depends on $w_i$ with $2^{-i} \ge \ell(Q)$. In particular, these notions are independent (meaning that
for any fixed $Q \in \mathcal{D}_0$ the random variable $w \mapsto 1_{\textup{good}}(Q+w)$ and any random variable that depends only on the cube $Q+w$ as a set, like $w \mapsto \int_{Q+w} f\,d\mu$, are independent).

\subsection{Beginning of the proof of the local $Tb$ theorem}
Fix a compactly supported function $f$.
Like in p. 3 of \cite{MM} we begin by writing the identity
\begin{align*}
\iint_{\R^{n+1}_+} |\theta_t f(x)|^2\,d\mu(x)\frac{dt}{t} = 
\frac{1}{\pi_{\textrm{good}}} E_w \sum_{R \in \mathcal{D}(w)_{\textup{good}}} \iint_{W_R} |\theta_t f(x)|^2\,d\mu(x)\frac{dt}{t},
\end{align*}
where $W_R = R \times (\ell(R)/2, \ell(R))$ is the Whitney region associated with $R \in \mathcal{D} = \mathcal{D}(w)$. This is based on the facts that for every fixed $R \in \mathcal{D}_0$
the random variables $1_{\textup{good}}(R+w)$ and $\iint_{W_{R+w}} |\theta_t f(x)|^2\,d\mu(x)\frac{dt}{t}$ are independent, and that we have $\pi_{\textrm{good}} = \mathbb{P}_{w}(R + w \textrm{ is good}) = E_w 1_{\textup{good}}(R+w)$.

We fix the grid $\mathcal{D} = \mathcal{D}(w)$ i.e. we fix $w$ from the probability space.
It is enough to prove that for any fixed large $s$ there holds that
\begin{displaymath}
\mathop{\sum_{R \in \mathcal{D}_{\textup{good}}}}_{\ell(R) \le 2^s}  \iint_{W_R} |\theta_t f(x)|^2\,d\mu(x)\frac{dt}{t} \lesssim \|f\|_{L^2(\mu)}^2.
\end{displaymath}
Now fix $N \in \N$ such that spt$\ f \subset B(0,2^N)$ and consider any $s \ge N$.
We expand
\begin{displaymath}
f = \mathop{\mathop{\sum_{Q_0 \in \mathcal{D}}}_{\ell(Q_0) = 2^s}}_{Q_0 \cap B(0, 2^N) \ne \emptyset} \mathop{\sum_{Q \in \mathcal{D}}}_{Q \subset Q_0} \Delta_Q f.
\end{displaymath}
Here $\Delta_Q$, $Q \subset Q_0$, are like in \eqref{eq:mdo}, but on the largest $Q_0$ level we agree (by abuse of notation) that $\Delta_{Q_0} = E_{Q_0} + \Delta_{Q_0}$, where
$E_{Q_0}f = \frac{\langle f \rangle_{Q_0}}{\langle b_{Q_0}\rangle_{Q_0}}b_{Q_0}$. Therefore, we have that $\int \Delta_Qf \,d\mu = 0$ except when $Q  = Q_0$ for some $Q_0$ with $\ell(Q_0) = 2^s$.
Since $\#\{Q_0 \in \mathcal{D}:\, Q_0 \cap B(0,2^N) \ne \emptyset\} \lesssim 1$, we can fix one $Q_0$ with $\ell(Q_0) = 2^s$, and concentrate on proving that
\begin{displaymath}
\mathop{\sum_{R \in \mathcal{D}_{\textup{good}}}}_{\ell(R) \le 2^s}  \iint_{W_R} \Big|\mathop{\sum_{Q \in \mathcal{D}}}_{Q \subset Q_0}\theta_t \Delta_Q f(x)\Big|^2\,d\mu(x)\frac{dt}{t} \lesssim \|f\|_{L^2(\mu)}^2.
\end{displaymath}

Using \cite{MM} heavily we see that we need to only deal with the case $\ell(Q) > 2^r \ell(R)$ and $d(Q,R) \le \ell(R)^{\gamma}\ell(Q)^{1-\gamma}$ presented in Subsection 2.7 of \cite{MM}.
This is because if we follow the splitting of the $Q$ summation from \cite{MM} we see that
the other parts of the summation (Subsections 2.4 to 2.6)  only need the fact that spt$\, \Delta_Q f \subset Q$, $\int \Delta_Q f d\mu = 0$ if $\ell(Q) < 2^s$, and $\sum_Q \|\Delta_Q f\|_{L^2(\mu)}^2 \lesssim \|f\|_{L^2(\mu)}^2$ (for the last
bound one can consult \cite{NTVa}).
\begin{rem}
This is in contrast with the Calder\'on--Zygmund world, where especially the diagonal part is extremely difficult in local $Tb$ theorems. In fact, so much so that Nazarov--Treil--Volberg \cite{NTVa} require an antisymmetric kernel $K$ to cope with the diagonal
part when they assume only BMO type bounds for $Tb_Q$. In our case the same argument as in Subsection 2.6 of \cite{MM} works here, since the diagonal is in fact
pretty trivial because of the rather strong size estimate of the kernels $s_t$. Indeed, the finer structure of the operators $\Delta_Q$ does not matter, since we do not need surgery unlike in \cite{NTVa}, and only use the square function bound $\sum_Q \|\Delta_Q f\|_{L^2(\mu)}^2 \lesssim \|f\|_{L^2(\mu)}^2$.
\end{rem}

In the case $\ell(Q) > 2^r \ell(R)$ and $d(Q,R) \le \ell(R)^{\gamma}\ell(Q)^{1-\gamma}$ one uses the goodness of $R$ to conclude that one must actually have that $R \subset Q$.
Therefore, things reduce to proving that
\begin{align}\label{eq:maineq}
 \mathop{\mathop{\sum_{R \in \mathcal{D}_{\textup{good}}}}_{\ell(R) < 2^{s-r}}}_{R \subset Q_0} \iint_{W_R} \Big| \sum_{k=r+1}^{s+\textup{gen}(R)} \theta_t \Delta_{R^{(k)}} f(x)\Big|^2\,d\mu(x)\frac{dt}{t} \lesssim \|f\|_{L^2(\mu)}^2,
\end{align}
where gen$(R)$ is determined by $\ell(R) = 2^{-\textup{gen}(R)}$, and $R^{(k)} \in \mathcal{D}$ is the unique cube for which
$\ell(R^{(k)}) = 2^k\ell(R)$ and $R \subset R^{(k)}$.

After these reductions the rest of the proof is focused on proving \eqref{eq:maineq}. In what follows we don't always write that now all of the cubes are inside $Q_0$.
\subsection{The case $(R^{(k-1)})^a = (R^{(k)})^a$}
In this case we may write
\begin{align}\label{eq:split1}
\Delta_{R^{(k)}} f = 1_{R^{(k)} \setminus R^{(k-1)}} \Delta_{R^{(k)}}f - 1_{(R^{(k-1)})^c} B_{R^{(k-1)}} b_{(R^{(k)})^a} + B_{R^{(k-1)}}b_{(R^{(k)})^a},
\end{align}
where
\begin{displaymath}
B_{R^{(k-1)}} = \frac{\langle f \rangle_{R^{(k-1)}}}{\langle b_{(R^{(k-1)})^a}\rangle_{R^{(k-1)}}} - \frac{\langle f \rangle_{R^{(k)}} }{\langle b_{(R^{(k)})^a}\rangle_{R^{(k)} }}
\end{displaymath}
with the minus term missing if $\ell(Q) = 2^s$.

If $S \in \textup{ch}(R^{(k)})$, $S \ne R^{(k-1)}$, and $(x,t) \in W_R$, we have by the size estimate \eqref{eq:size} that
\begin{align*}
|\theta_t(1_S \Delta_{R^{(k)}}f)(x)| &\lesssim \int_S \frac{\ell(R)^{\alpha}}{d(S,R)^{m+\alpha}} |\Delta_{R^{(k)}} f(y)|\,d\mu(y) \\
&\lesssim \int_S \Big(\frac{\ell(R)}{\ell(S)}\Big)^{\alpha/2}\frac{1}{\ell(S)^m} |\Delta_{R^{(k)}} f(y)|\,d\mu(y)\\ &\lesssim 2^{-\alpha k/2} \mu(R^{(k-1)})^{-1/2} \|\Delta_{R^{(k)}} f\|_{L^2(\mu)}.
\end{align*}
Here we used that by goodness $d(R, S) \ge \ell(R)^{\gamma}\ell(S)^{1-\gamma}$. Therefore, we have that
\begin{displaymath}
|\theta_t(1_{R^{(k)} \setminus R^{(k-1)}} \Delta_{R^{(k)}}f)(x)| \lesssim 2^{-\alpha k/2} \mu(R^{(k-1)})^{-1/2} \|\Delta_{R^{(k)}} f\|_{L^2(\mu)}, \qquad (x,t) \in W_R.
\end{displaymath}

Accretivity condition gives that
\begin{align*}
|B_{R^{(k-1)}}| \mu(R^{(k-1)}) \lesssim \Big| \int_{R^{(k-1)}} B_{R^{(k-1)}} b_{(R^{(k)})^a}\,d\mu\Big| &= \Big| \int_{R^{(k-1)}} \Delta_{R^{(k)}} f\,d\mu\Big|  \\
&\lesssim \mu(R^{(k-1)})^{1/2}  \|\Delta_{R^{(k)}} f\|_{L^2(\mu)}.
\end{align*}
We also have by the size condition \eqref{eq:size}, and the fact that $\|b_{(R^{(k)})^a}\|_{L^{\infty}(\mu)} \lesssim 1$, that
$|\theta_t( 1_{(R^{(k-1)})^c} b_{(R^{(k)})^a})(x)|$ can be dominated by
\begin{align*}
\ell(R)^{\alpha} \int_{\R^n \setminus B(x, d(R, \R^n \setminus R^{(k-1)}))}  \frac{d\mu(y)}{|x-y|^{m+\alpha}}
\lesssim \ell(R)^{\alpha}d(R, \R^n \setminus R^{(k-1)})^{-\alpha} \lesssim 2^{-\alpha k /2}.
\end{align*}
Here goodness was used to conclude that $d(R, \R^n \setminus R^{(k-1)}) \ge \ell(R)^{1/2}\ell(R^{(k-1)})^{1/2}$. Therefore, we have that
\begin{displaymath}
|\theta_t( 1_{(R^{(k-1)})^c}B_{R^{(k-1)}}  b_{(R^{(k)})^a})(x)| \lesssim 2^{-\alpha k/2} \mu(R^{(k-1)})^{-1/2} \|\Delta_{R^{(k)}} f\|_{L^2(\mu)}, \qquad (x,t) \in W_R.
\end{displaymath}
This is the same bound as for the previous term. We note that
\begin{align*}
\sum_{R:\, \ell(R) < 2^{s-r}} \mu(R) \Big[ \sum_{k=r+1}^{s+\textup{gen}(R)} 2^{-\alpha k /2} \mu(R^{(k-1)})^{-1/2} \|\Delta_{R^{(k)}} f\|_{L^2(\mu)} \Big]^2 \lesssim \|f\|_{L^2(\mu)}^2.
\end{align*}
This is an exercise in summation (or see the very last lines of \cite{MM}). Therefore, the first two terms of the splitting \eqref{eq:split1} are in control. The last term of \eqref{eq:split1}, that is $B_{R^{(k-1)}}b_{(R^{(k)})^a}$, will become part of the paraproduct
to be dealt with later.

\subsection{The case $(R^{(k-1)})^a = R^{(k-1)}$} This time we begin by simply writing
\begin{align*}
\Delta_{R^{(k)}} f = 1_{R^{(k)} \setminus R^{(k-1)}} \Delta_{R^{(k)}}f + 1_{R^{(k-1)}} \Delta_{R^{(k)}}f.
\end{align*}
The first term is in check by the argument above. We then decompose
\begin{align*}
 1_{R^{(k-1)}} \Delta_{R^{(k)}}f = \frac{\langle f \rangle_{R^{(k-1)}}}{\langle b_{R^{(k-1)}}\rangle_{R^{(k-1)}}}&b_{R^{(k-1)}} - \frac{\langle f \rangle_{R^{(k)}} }{\langle b_{(R^{(k)})^a}\rangle_{R^{(k)} }}b_{(R^{(k)})^a} \\
 &+ 1_{(R^{(k-1)})^c} \frac{\langle f \rangle_{R^{(k)}} }{\langle b_{(R^{(k)})^a}\rangle_{R^{(k)} }}b_{(R^{(k)})^a}.
\end{align*}
For the last term we have from above that $|\theta_t( 1_{(R^{(k-1)})^c} b_{(R^{(k)})^a})(x)| \lesssim 2^{-\alpha k /2}$, if $(x,t) \in W_R$.
The term in front is simply estimated using the construction of the stopping time:
\begin{displaymath}
\frac{|\langle f \rangle_{R^{(k)}}| }{|\langle b_{(R^{(k)})^a}\rangle_{R^{(k)} }|} \lesssim |\langle f \rangle_{R^{(k)}}|.
\end{displaymath}
To finish the estimation of this term we may then use the bound
\begin{align*}
\sum_{R:\, \ell(R) < 2^{s-r}} \mu(R) \Big[& \mathop{\sum_{k=r+1}^{s+\textup{gen}(R)}}_{(R^{(k-1)})^a = R^{(k-1)}} 2^{-\alpha k /2} |\langle f \rangle_{R^{(k)}}| \Big]^2 \\
& \lesssim \sum_{S: \, \ell(S) \le 2^s} A_S |\langle f \rangle_S|^2 \lesssim \|f\|_{L^2(\mu)}^2,
\end{align*}
where the last bound follows since \begin{displaymath}
A_S := \mathop{\sum_{S' \in \textup{ch}(S)}}_{(S')^a = S'} \mu(S')
\end{displaymath}
is a Carleson sequence by Lemma \ref{lem:mescar}. The rest will again become part of the paraproduct, which we will deal with in the next subsection.

\subsection{The Carleson estimate for the paraproduct}
Combining the above two cases and collapsing the remaining telescoping summation we have reduced to estimating 
\begin{align*}
\mathop{\mathop{\sum_{R \in \mathcal{D}_{\textup{good}}}}_{\ell(R) < 2^{s-r}}}_{R \subset Q_0}& \iint_{W_R} \Big| \frac{\langle f \rangle_{R^{(r)}}}{\langle b_{(R^{(r)})^a} \rangle_{R^{(r)}}} \theta_t b_{(R^{(r)})^a}(x)
\Big|^2\,d\mu(x)\frac{dt}{t} \\
&\lesssim \sum_S |\langle f \rangle_S|^2 \mathop{\sum_{R \in \mathcal{D}_{\textup{good}}}}_{S = R^{(r)}}\iint_{W_R} |\theta_t b_{S^a}(x)|^2 d\mu(x)\frac{dt}{t} =: \sum_S B_S|\langle f \rangle_S|^2.
\end{align*}
The proof of the estimate \eqref{eq:SFbound}, and thus of the local $Tb$ theorem, Theorem \ref{localTb}, is completed by the next lemma.
\begin{lem}
There holds for every $R \in \mathcal{D}$ that
\begin{displaymath}
\sum_{S \subset R} B_S \lesssim \mu(R).
\end{displaymath}
\end{lem}
\begin{proof}
Fix $R \in \mathcal{D}$. We have that
\begin{align*}
&\sum_{S \subset R} B_S = \sum_{S \subset R} \mathop{\sum_{Q \in \mathcal{D}_{\textup{good}}}}_{S = Q^{(r)}}\iint_{W_Q} |\theta_t b_{S^a}(x)|^2 d\mu(x)\frac{dt}{t} \\
&= \Big(\mathop{\sum_{S \subset R}}_{S^a = R^a} + \mathop{\sum_{H \subset R}}_{H^a = H} \sum_{S:\,S^a = H}\Big)\mathop{\sum_{Q \in \mathcal{D}_{\textup{good}}}}_{S = Q^{(r)}}\iint_{W_Q} |\theta_t b_{S^a}(x)|^2 d\mu(x)\frac{dt}{t}.
\end{align*}
By Lemma \ref{lem:mescar} it is enough to prove that for an arbitrary $H \in \mathcal{D}$ there holds that
\begin{align}\label{eq:CarRed}
J(H) := \mathop{\sum_{S \subset H}}_{S^a = H^a} \mathop{\sum_{Q \in \mathcal{D}_{\textup{good}}}}_{S = Q^{(r)}}\iint_{W_Q} |\theta_t b_{H^a}(x)|^2 d\mu(x)\frac{dt}{t} \lesssim \mu(H).
\end{align}
Indeed, assuming \eqref{eq:CarRed} we can complete the proof by noting that now
\begin{displaymath}
\mathop{\sum_{S \subset R}}_{S^a = R^a} \mathop{\sum_{Q \in \mathcal{D}_{\textup{good}}}}_{S = Q^{(r)}}\iint_{W_Q} |\theta_t b_{R^a}(x)|^2 d\mu(x)\frac{dt}{t} \lesssim \mu(R)
\end{displaymath}
and
\begin{displaymath}
\mathop{\sum_{H \subset R}}_{H^a = H} \sum_{S:\,S^a = H} \mathop{\sum_{Q \in \mathcal{D}_{\textup{good}}}}_{S = Q^{(r)}}\iint_{W_Q} |\theta_t b_{H}(x)|^2 d\mu(x)\frac{dt}{t} \lesssim \mathop{\sum_{H \subset R}}_{H^a = H} \mu(H) \lesssim \mu(R),
\end{displaymath}
where in the last estimate we use Lemma \ref{lem:mescar}.

We will then prove \eqref{eq:CarRed}. To this end, fix an $H \in \mathcal{D}$.
Let $\mathcal{F}(H)$ consist of the maximal cubes $Q$ such that $\ell(Q) \le 2^{-r}\ell(H)$ and $d(Q,H^c) \ge 3\ell(Q)$. We have by goodness (to get the bound $d(Q, H^c) \ge 3\ell(Q)$)
and the property (4) of $b_{H^a}$ that
\begin{align*}
J(H) &\lesssim \mathop{\mathop{\sum_{Q \in \mathcal{D}}}_{\ell(Q) \le 2^{-r}\ell(H)}}_{d(Q,H^c) \ge 3\ell(Q)}\iint_{W_Q} |\theta_t b_{H^a}(x)|^2 d\mu(x)\frac{dt}{t} \\
&\lesssim \sum_{Q \in \mathcal{F}(H)}  \iint_{\widehat Q} |\theta_t b_{H^a}(x)|^2 d\mu(x)\frac{dt}{t} \lesssim  \sum_{Q \in \mathcal{F}(H)} \mu(3Q) \lesssim \mu(H).
\end{align*}
The last estimate follows from $\sum_{Q \in \mathcal{F}(H)} 1_{3Q} \lesssim 1_H$.
\end{proof}

\begin{rem}
The local $Tb$ theorem, Theorem \ref{localTb}, can be proved assuming only that $\mu(B(x,r)) \le \lambda(x,r)$ for some $\lambda\colon \R^n \times (0,\infty) \to (0,\infty)$ satisfying that
$r \mapsto \lambda(x,r)$ is non-decreasing and $\lambda(x, 2r) \le C_{\lambda}\lambda(x,r)$ for all $x \in \R^n$ and $r > 0$. In this case one only needs to replace the kernel estimates by
\begin{displaymath}
|s_t(x,y)| \lesssim \frac{t^{\alpha}}{t^{\alpha}\lambda(x,t) + |x-y|^{\alpha}\lambda(x, |x-y|)}
\end{displaymath}
and
\begin{displaymath}
|s_t(x,y) - s_t(x,z)| \lesssim \frac{|y-z|^{\alpha}}{t^{\alpha}\lambda(x,t) + |x-y|^{\alpha}\lambda(x, |x-y|)}
\end{displaymath}
whenever $|y-z| < t/2$.
\end{rem}

\section{Endpoint and $L^p(\mu)$ theory for $S$ and $V$}

In this section, we prove Theorem \ref{LpSF}.

\subsection{The $L^{\infty}(\mu) \to L^{1,\infty}(\mu)$ bound for $S$ and $V$}
\begin{thm}\label{nonhomweakbound}
If $S\colon L^2(\mu) \to L^2(\mu)$ boundedly, then $S\colon L^1(\mu) \to L^{1,\infty}(\mu)$ boundedly. The same implication holds for $V$.
\end{thm}
\begin{proof} The proofs for $S$ and $V$ are quite similar, the one for $V$ being a bit simpler. Consequently, we give the full details only for $S$ and indicate by several in-proof remarks how to make the relevant changes for $V$. Both proofs start by recalling the non-homogeneous Calder\'on-Zygmund decomposition of X. Tolsa.

Let $f \in L^1(\mu)$ and $\lambda > 0$. If $\mu(\R^n) = \infty$ (which we assume for convenience), see \cite{To2}, there exists a family of finitely overlapping
cubes $Q_i$ such that $\int_{Q_i} |f| \,d\mu \gtrsim \lambda \mu(2Q_i)$, $\int_{\eta Q_i} |f|\,d\mu \lesssim \lambda \mu(2\eta Q_i)$, $\eta \ge 2$, and $|f| \le \lambda$ holds $\mu$-a.e. on $\R^n \setminus \bigcup Q_i$.
Let $R_i$ be the smallest $(6, 6^{m+1})$-doubling cube of the form $6^kQ_i$, $k \ge 1$. Set $w_i = 1_{Q_i} / \sum_k 1_{Q_k}$. There exists functions $\varphi_i$ satisfying spt$\,\varphi_i \subset R_i$,
$\int \varphi_i \,d\mu = \int_{Q_i} fw_i\,d\mu$, $\sum |\varphi_i| \lesssim \lambda$ and $\|\varphi_i\|_{L^{\infty}(\mu)} \mu(R_i) \lesssim \int_{Q_i} |f|\,d\mu$.
Finally, we have the decomposition $f = g + b$, $g = f1_{\R^n \setminus \bigcup Q_i} + \sum \varphi_i$, $b = \sum b_i$, $b_i = fw_i - \varphi_i$.

The subadditivity and $L^2(\mu)$ boundedness of $S$ together with the properties of the decomposition show that it is enough to prove that for every $i$ there holds that
\begin{displaymath}
 \int_{(4R_i)^c} Sb_i(x)\,d\mu(x) + \int_{4R_i \setminus 2Q_i} S(fw_i)(x)\,d\mu(x) \lesssim \int_{Q_i} |f|\,d\mu.
\end{displaymath}
The whole thing boils down to pointwise estimates, which are more difficult to obtain for square functions than for Calder\'on--Zygmund operators. Otherwise, the proof structures agree.

We begin by proving the bound 
\begin{displaymath}
 \int_{(4R_i)^c} Sb_i(x)\,d\mu(x) \lesssim \int_{Q_i} |f|\,d\mu.
\end{displaymath}
For this, it is enough to prove that for a fixed $x \in (4R_i)^c$ there holds that
\begin{equation}\label{bestimate}
Sb_i(x) \lesssim \Big(\frac{\ell(R_i)^{\alpha}}{|x-c_{R_i}|^{m+\alpha}} +  \frac{\ell(R_i)^{\alpha/2}}{|x-c_{R_i}|^{m+\alpha/2}}\Big)\|b_i\|_{L^1(\mu)}.
\end{equation}
Down until here, the proofs for $S$ and $V$ are identical, but now we need to concentrate a moment on $S$ alone. 

\subsubsection{Estimating $Sb_{i}(x)$} We bound
\begin{align}
\label{Sspecific} Sb_i(x) &\le \Big( \iint_{\Gamma(x) \cap [2R_i \times \R_+]} |\theta_t b_i(y)|^2 \,\frac{d\mu(y) dt}{t^{m+1}}\Big)^{1/2} \\
&+ \Big( \iint_{\Gamma(x) \cap [(2R_i)^c \times (0, \ell(R_i))]} |\theta_t b_i(y)|^2 \,\frac{d\mu(y) dt}{t^{m+1}}\Big)^{1/2} \notag\\
&+  \Big( \iint_{\Gamma(x) \cap [(2R_i)^c \times [\ell(R_i), \infty)]} |\theta_t b_i(y)|^2 \,\frac{d\mu(y) dt}{t^{m+1}}\Big)^{1/2} = I + II + III. \notag \end{align}

Notice that in the term $I$ there holds that
\begin{displaymath}
t > |x-y| \ge \ell(R_i) \ge 2|z-c_{R_i}|
\end{displaymath}
for every $z \in R_i$. Using $\int b_i\,d\mu = 0$ we have by the H\"older estimate for $s_t$ that
\begin{align*}
|\theta_t b_i(y)| \le \int_{R_i} &|s_t(y,z) - s_t(y,c_{R_i})| |b_i(z)|\,d\mu(z)\\ & \lesssim \frac{\ell(R_i)^{\alpha}}{t^{m+\alpha}} \|b_i\|_{L^1(\mu)} \lesssim \frac{\ell(R_i)^{\alpha}}{|x-c_{R_i}|^{m+\alpha}} \|b_i\|_{L^1(\mu)}.
\end{align*}
The last estimate follows since $t > |x-y| \ge |x-c_{R_i}|/2$. We now have that
\begin{displaymath}
I \lesssim \Big(  \mu(2R_i) \int_{\ell(R_i)}^{\infty}   \frac{dt}{t^{m+1}}\Big)^{1/2}   \frac{\ell(R_i)^{\alpha}}{|x-c_{R_i}|^{m+\alpha}} \|b_i\|_{L^1(\mu)} \lesssim  \frac{\ell(R_i)^{\alpha}}{|x-c_{R_i}|^{m+\alpha}} \|b_i\|_{L^1(\mu)}.
\end{displaymath}

In the term $II$ we note that since $y \in (2R_i)^c$ and $|x-y| < t < \ell(R_i) \le |x-c_{R_i}|/2$ we have for every $z \in R_i$ that
\begin{displaymath}
|y-z| \gtrsim |y-c_{R_i}| \gtrsim |x-c_{R_i}|.
\end{displaymath}
Thus, the size estimate of $s_t$ gives that
\begin{displaymath}
|\theta_t b_i(y)| \lesssim \frac{t^{\alpha}}{|x-c_{R_i}|^{m+\alpha}} \|b_i\|_{L^1(\mu)}.
\end{displaymath}
This yields that
\begin{align*}
II \lesssim \Big( \int_0^{\ell(R_i)}& t^{2\alpha-1} t^{-m}\mu(B(x,t))  \Big)^{1/2}\frac{1}{|x-c_{R_i}|^{m+\alpha}} \|b_i\|_{L^1(\mu)} \\ 
&\lesssim \Big( \int_0^{\ell(R_i)} t^{2\alpha-1} \,dt\Big)^{1/2} \frac{1}{|x-c_{R_i}|^{m+\alpha}} \|b_i\|_{L^1(\mu)} \lesssim \frac{\ell(R_i)^{\alpha}}{|x-c_{R_i}|^{m+\alpha}} \|b_i\|_{L^1(\mu)}.
\end{align*}

In the term $III$ we note that either $t > |x-c_{R_i}|/2$ or $|y-c_{R_i}| > |x-c_{R_i}|/2$. Indeed, otherwise there holds that
\begin{displaymath}
\frac{1}{2}|x-c_{R_i}| \le |x-c_{R_i}| - |y-c_{R_i}| \le |x-y| < t \le \frac{1}{2}|x-c_{R_i}|,
\end{displaymath}
which is a contradiction. Therefore, noting that $\ell(R_i) \le |x-c_{R_i}|/2$, we may bound
\begin{align*}
III \le \Big( \int_{\ell(R_i)}^{|x-c_{R_i}|/2}& \int_{B(x,t) \cap B(c_{R_i}, |x-c_{R_i}|/2)^c} |\theta_t b_i(y)|^2 \,\frac{d\mu(y) dt}{t^{m+1}}\Big)^{1/2} \\
&+ \Big( \int_{|x-c_{R_i}|/2}^{\infty} \int_{B(x,t)} |\theta_t b_i(y)|^2 \,\frac{d\mu(y) dt}{t^{m+1}}\Big)^{1/2} = III' + III''.
\end{align*}

Since $t \ge \ell(R_i)$, we may use the H\"older estimate of $s_t$ to the effect that
\begin{displaymath}
|\theta_t b_i(y)| \lesssim \frac{\ell(R_i)^{\alpha}}{(t+|y-c_{R_i}|)^{m+\alpha}} \|b_i\|_{L^1(\mu)}.
\end{displaymath}
In $III'$ we also bound $\ell(R_i)^{\alpha} \le t^{\alpha/2}\ell(R_i)^{\alpha/2}$ and $|y-c_{R_i}| \gtrsim |x-c_{R_i}|$ so that
\begin{displaymath}
III' \lesssim \Big(\int_{0}^{|x-c_{R_i}|/2} t^{\alpha-1}\,dt \Big)^{1/2} \frac{\ell(R_i)^{\alpha/2}}{|x-c_{R_i}|^{m+\alpha}} \|b_i\|_{L^1(\mu)}
\lesssim \frac{\ell(R_i)^{\alpha/2}}{|x-c_{R_i}|^{m+\alpha/2}} \|b_i\|_{L^1(\mu)}.
\end{displaymath}
For $III''$ we have that
\begin{displaymath}
III'' \lesssim \Big( \int_{|x-c_{R_i}|/2}^{\infty} t^{-2m-2\alpha-1} \,dt\Big)^{1/2} \ell(R_i)^{\alpha} \|b_i\|_{L^1(\mu)} \lesssim \frac{\ell(R_i)^{\alpha}}{|x-c_{R_i}|^{m+\alpha}} \|b_i\|_{L^1(\mu)}.
\end{displaymath}

\subsubsection{Modifications for $V$, part one} In the proof for $V$, the splitting, which starts on line \eqref{Sspecific}, is replaced by the following:
\begin{align*} Vb_{i}(x) & \leq \left( \int_{0}^{\ell(R_{i})} |\theta_{t}b_{i}(x)|^{2} \, \frac{dt}{t} \right)^{1/2}\\
& + \left(\int_{\ell(R_{i})}^{|x - c_{R_{i}}|} |\theta_{t}b_{i}(x)|^{2} \, \frac{dt}{t} \right)^{1/2}\\
& + \left(\int_{|x - c_{R_{i}}|}^{\infty} |\theta_{t}b_{i}(x)|^{2} \, \frac{dt}{t} \right)^{1/2} =: I_{V} + II_{V} + III_{V}.   \end{align*} 
The piece $III_{V}$ here corresponds to the piece $I$ above: using the $y$-continuity of the kernel $s_{t}$, one arrives at
\begin{displaymath} \frac{|\theta_{t}b_{i}(x)|^{2}}{t} \lesssim \frac{\ell(R_{i})^{2\alpha}}{t^{2m + 2\alpha + 1}}\|b_{i}\|_{L^{1}(\mu)}^{2}, \end{displaymath} 
and integrating over $t \geq |x - c_{R_{i}}|$ gives the desired bound.

In $I_{V}$, one notes that $|x - y| \sim |x - c_{R_{i}}|$ for $y \in R_{i}$, whence 
\begin{displaymath} \frac{|\theta_{t}b_{i}(x)|^{2}}{t} \lesssim \frac{t^{2\alpha - 1}}{|x - c_{R_{i}}|^{2m + 2\alpha}}\|b_{i}\|_{L^{1}(\mu)}^{2}. \end{displaymath} 
An integration over $0 \leq t \leq \ell(R_{i})$ finishes the estimate for $I_{V}$. 

In $II_{V}$, we again use the $y$-continuity of $s_{t}$, combined with $\ell(R_{i})^{\alpha} \lesssim \ell(R_{i})^{\alpha/2}t^{\alpha/2}$ and $|x - y| \sim |x - c_{R_{i}}|$. The result is
\begin{displaymath} \frac{|\theta_{t}b_{i}(x)|^{2}}{t} \lesssim \frac{\ell(R_{i})^{\alpha} \cdot t^{\alpha - 1}}{|x - c_{R_{i}}|^{2m + 2\alpha}}\|b_{i}\|_{L^{1}(\mu)}^{2}. \end{displaymath}
Integrating this expression over $0 \leq t \leq |x - c_{R_{i}}|$ gives the right bound and shows that \eqref{bestimate} holds with $S$ replaced by $V$.

\subsubsection{Back to $S$} To finish the proof for $S$, it remains to show that
\begin{displaymath}
\int_{4R_i \setminus 2Q_i} S(fw_i)(x)\,d\mu(x) \lesssim \int_{Q_i} |f|\,d\mu.
\end{displaymath}
We will prove that for a fixed $x \in 4R_i \setminus 2Q_i$ there holds that
\begin{equation}\label{Sestimate}
S(fw_i)(x) \lesssim \frac{1}{|x-c_{Q_i}|^m} \int_{Q_i} |f|\,d\mu.
\end{equation}
This is enough, since
\begin{displaymath}
\int_{4R_i \setminus 2Q_i} \frac{d\mu(x)}{|x-c_{Q_i}|^m} \le \Big(\int_{4R_i \setminus R_i} + \int_{R_i \setminus 6Q_i} + \int_{6Q_i \setminus Q_i}\Big) \frac{d\mu(x)}{|x-c_{Q_i}|^m} \lesssim 1,
\end{displaymath}
where the first and last terms are trivial to estimate (recall $c_{Q_i} = c_{R_i}$). The estimate for the middle term follows from a standard calculation recalling that
there are no $(6, 6^{m+1})$-doubling cubes of the form $6^kQ_i$ strictly between  $6Q_i$ and $R_i$.

So fix $x \in 4R_i \setminus 2Q_i$. We estimate
\begin{align*}
S(fw_i)(x) &\le \Big( \iint_{\Gamma(x) \cap [1.5Q_i \times \R_+]} |\theta_t (fw_i)(y)|^2 \,\frac{d\mu(y) dt}{t^{m+1}}\Big)^{1/2} \\
&+  \Big( \iint_{\Gamma(x) \cap [(1.5Q_i)^c \times (0, |x-c_{Q_i}|/2)]} |\theta_t (fw_i)(y)|^2 \,\frac{d\mu(y) dt}{t^{m+1}}\Big)^{1/2} \\
&+ \Big( \iint_{\Gamma(x) \cap [(1.5Q_i)^c \times [|x-c_{Q_i}|/2, \infty)]} |\theta_t (fw_i)(y)|^2 \,\frac{d\mu(y) dt}{t^{m+1}}\Big)^{1/2} = A + B + C.
\end{align*}
Let $z \in Q_i$. We use the size estimate.
In $A$ we estimate
\begin{displaymath}
|s_t(y,z)| \lesssim \frac{1}{t^m} \lesssim \frac{1}{|x-c_{Q_i}|^m}.
\end{displaymath}
In $B$ we have
\begin{displaymath}
|s_t(y,z)| \lesssim \frac{t^{\alpha}}{|y-z|^{m+\alpha}} \lesssim \frac{t^{\alpha}}{|y-c_{Q_i}|^{m+\alpha}} \lesssim \frac{t^{\alpha}}{|x-c_{Q_i}|^{m+\alpha}}.
\end{displaymath}
Lastly, in $C$ we simply use
\begin{displaymath}
|s_t(y,z)| \lesssim \frac{1}{t^m}.
\end{displaymath}

We now have:
\begin{align*}
A \lesssim \Big(\mu(1.5Q_i) \int_{\ell(Q_i)/4}^{\infty} \frac{dt}{t^{m+1}} \Big)^{1/2}\frac{1}{|x-c_{Q_i}|^m} \int_{Q_i} |f|\,d\mu \lesssim \frac{1}{|x-c_{Q_i}|^m} \int_{Q_i} |f|\,d\mu,
\end{align*}
\begin{align*}
B \lesssim \Big( \int_0^{|x-c_{Q_i}|/2}  t^{2\alpha-1}\,dt \Big)^{1/2} \frac{1}{|x-c_{Q_i}|^{m+\alpha}} \int_{Q_i} |f|\,d\mu \lesssim  \frac{1}{|x-c_{Q_i}|^m} \int_{Q_i} |f|\,d\mu
\end{align*}
and
\begin{align*}
C \lesssim \Big( \int_{|x-c_{Q_i}|/2}^{\infty} t^{-2m-1}\,dt \Big)^{1/2} \int_{Q_i} |f|\,d\mu \lesssim \frac{1}{|x-c_{Q_i}|^m} \int_{Q_i} |f|\,d\mu.
\end{align*}
This completes the proof for $S$.
\subsubsection{Modifications for $V$, part two} To complete the proof for $V$, it suffices to prove \eqref{Sestimate} with $S$ replaced by $V$. The splitting is very natural:
\begin{displaymath} V(fw_{i})(x) \leq \left( \int_{0}^{|x - c_{Q_{i}}|} |\theta_{t}(fw_{i})(x)|^{2} \, \frac{dt}{t} \right)^{1/2} + \left( \int_{|x - c_{Q_{i}}|}^{\infty} |\theta_{t}(fw_{i})(x)|^{2} \, \frac{dt}{t} \right)^{1/2}. \end{displaymath}
These terms are estimated exactly like $B$ and $C$ above, and there is no point in repeating the details.

The $L^{1}(\mu) \to L^{1,\infty}(\mu)$ bounds are now established for $S$ and $V$ alike.
\end{proof}

\begin{cor}
If $S\colon L^2(\mu) \to L^2(\mu)$ boundedly, then $S\colon L^p(\mu) \to L^p(\mu)$ boundedly for $1 < p < 2$.
\end{cor}

\subsection{$L^{\infty}(\mu) \to \textup{RBMO}(\mu)$ and $L^p(\mu) \to L^p(\mu)$ for $2 < p < \infty$}
In this subsection we will prove the next theorem.

\begin{thm}\label{thm:rbmo} Assume that the square function $S$ is a bounded mapping $L^{2}(\mu) \to L^{2}(\mu)$. Then $S$ is a bounded mapping $L^{\infty}(\mu) \to \textup{RBMO}(\mu)$.
\end{thm}

\begin{cor} If $S \colon L^{2}(\mu) \to L^{2}(\mu)$ boundedly, then $S \colon L^{p}(\mu) \to L^{p}(\mu)$ boundedly for all $2 < p < \infty$. \end{cor}
\begin{proof} This follows by combining Theorem \ref{thm:rbmo} with the interpolation result for sublinear operators in \cite{LY}. \end{proof}

To prove Theorem \ref{thm:rbmo}, we start with a technical lemma.

\begin{lem} Let $a \in \R^{n}$, $r > 0$. Let $f \in L^{\infty}(\mu)$ with $\|f\|_{L^{\infty}(\mu)} \leq 1$. Then
\begin{displaymath} |S(f1_{\R^{n} \setminus B(a,Cr)})(x) - S(f1_{\R^{n} \setminus B(a,Cr)})(a)| \lesssim 1, \qquad x \in B(a,r), \end{displaymath}
for all large enough constants $C \geq 1$.
\end{lem} 

\begin{proof} Writing $S := S(f1_{\R^{n} \setminus B(a,Cr)})$, and using $|S(x) - S(a)| \leq |S^{2}(x) - S^{2}(a)|^{1/2}$, it suffices to prove that $|S^{2}(x) - S^{2}(a)| \lesssim 1$. The first thing to check is the following: for every $x \in B(a,r)$, one has
\begin{equation}\label{form1} \int_{0}^{2r} \int_{B(x,t)} |\theta_{t}(f1_{\R^{n} \setminus B(a,Cr)})(y)|^{2} \, d\mu(y) \, \frac{dt}{t^{m + 1}} \lesssim 1. \end{equation} 
Fix the parameters $x \in B(a,r)$, $t \leq 2r$, $y \in B(x,t)$, and consider
\begin{displaymath} |\theta_{t}(f1_{\R^{n} \setminus B(a,Cr)})(y)| \lesssim \int_{\R^{n} \setminus B(a,Cr)} \frac{t^{\alpha}}{(t + |y - z|)^{m + \alpha}} \, d\mu(z). \end{displaymath}
For $z \in \R^{n} \setminus B(a,Cr)$ and $y \in B(x,t) \subset B(a,4r)$ we have $|y - z| \sim |a - z|$, if $C$ is large enough. Thus, we have that
\begin{displaymath}
 |\theta_{t}(f1_{\R^{n} \setminus B(a,Cr)})(y)| \lesssim t^{\alpha} \int_{\R^{n} \setminus B(a,Cr)} \frac{d\mu(z)}{|a - z|^{m + \alpha}} \lesssim \frac{t^{\alpha}}{r^{\alpha}}. 
\end{displaymath} 
This estimate is uniform in $y \in B(x,t)$, so the left hand side of \eqref{form1} is bounded by a constant multiple of
\begin{displaymath} \int_{0}^{2r} \int_{B(x,t)} \frac{t^{2\alpha}}{r^{2\alpha}} \, d\mu(y) \, \frac{dt}{t^{m + 1}} \lesssim \frac{1}{r^{2\alpha}} \int_{0}^{2r} t^{2\alpha - 1} \, dt \sim 1. \end{displaymath}

Next, if $g:=f1_{\R^{n} \setminus B(a,Cr)}$ using \eqref{form1} write
\begin{align*} & |S^{2}(x) - S^{2}(a)|  \\
&\lesssim \left|\int_{2r}^{\infty} \int_{B(x,t)}  |\theta_{t}(g)(y)|^{2} \, d\mu(y) \, \frac{dt}{t^{m + 1}} - \int_{2r}^{\infty} \int_{B(a,t)}  |\theta_{t}(g)(y)|^{2} \, d\mu(y) \, \frac{dt}{t^{m + 1}} \right| + 1 \\
 &\leq \int_{2r}^{\infty} \int_{B(x,t) \Delta B(a,t)}  |\theta_{t}(g)(y)|^{2} \, d\mu(y) \, \frac{dt}{t^{m + 1}} + 1\\
 &\lesssim \int_{2r}^{\infty} \int_{B(x,t) \Delta B(a,t)} \, d\mu(y) \, \frac{dt}{t^{m + 1}} + 1. \end{align*}
The symmetric difference $B(x,t) \Delta B(a,t)$ is the union of $B(a,t) \setminus B(x,t)$ and $B(x,t) \setminus B(a,t)$, and we deal with the corresponding integrals separately. For instance, the part with $B(a,t) \setminus B(x,t)$ is estimated as follows:
\begin{align*}
\int_{2r}^{\infty} \int_{B(a,t) \setminus B(x,t)} \, d\mu(y) \, \frac{dt}{t^{m + 1}}  \leq \int_{\R^{n} \setminus B(a,r)} & \int_{|y - x|-r}^{|y - x|} \frac{dt}{|y-a|^{m + 1}} \, d\mu(y)\\
&\sim r\int_{\R^{n} \setminus B(a,r)} \frac{d\mu(y)}{|y-a|^{m + 1}} \sim 1.
\end{align*} 
The integral over the domain $B(x,t) \setminus B(a,t)$ is treated similarly, and the proof of the lemma is complete. \end{proof}

\begin{proof}[Proof of Theorem \ref{thm:rbmo}]
Fix $f \in L^{\infty}(\mu)$ with $\|f\|_{L^{\infty}} \leq 1$. With the previous lemma in use, the first condition of $Sf \in \textup{RBMO}(\mu)$ is straightforward to verify. Given any ball $B = B(a,r)$ set
\begin{displaymath} (Sf)_{B} := S(f1_{\R^{n} \setminus B(a,Cr)})(a), \end{displaymath}
where $C \geq 1$ is the constant from the lemma. Then, using the lemma, the sublinearity of $S$, and the $L^{2}(\mu) \to L^{2}(\mu)$ bound for $S$, one has the following estimate:
\begin{align*} \int_{B} |Sf(x) - (Sf)_{B}| \, d\mu(x) & \lesssim \int_{B} |Sf(x) - S(f1_{\R^{n} \setminus B(a,Cr)})(x)| \, d\mu(x) + \mu(B)\\
& \leq \int_{B} S(f1_{B(a,Cr)})(x) \, d\mu(x) + \mu(B)\\
& \leq \mu(B)^{1/2} \left(\int S(f1_{B(a,Cr)})(x)^{2} \, d\mu(x) \right)^{1/2} + \mu(B)\\
& \lesssim \mu(B)^{1/2}\|f1_{B(a,Cr)}\|_{L^{2}(\mu)} + \mu(B) \lesssim \mu(CB). \end{align*} 
This is precisely the first of the two conditions required for $Sf \in RBMO(\mu)$.

Verifying the second condition of $Sf \in \textup{RBMO}(\mu)$ is also quite easy. Let us denote by $r(B)$ the radius of any ball $B \subset \R^n$. Fix two balls $B \subset R \subset \R^{n}$ and let $a$ be the center of $B$. Recall that one is supposed to verify the estimate
\begin{equation}\label{rbmo2} |(Sf)_{R} - (Sf)_{B}| \lesssim 1 + \int_{CR \setminus B} \frac{d\mu(y)}{|y - a|^{m}}. \end{equation}
First of all, the previous lemma shows that
\begin{displaymath} |(Sf)_{R} - (Sf)_{B}| \lesssim 1 + |S(f1_{\R^{n} \setminus CR})(a) - S(f1_{\R^{n} \setminus CB})(a)|. \end{displaymath}
Next, choose an increasing sequence of balls $B_{0} \subset B_{1} \subset \ldots \subset B_{K}$ with the following properties:
\begin{itemize}
\item[(i)] $B_{0} = B$ and $B_{K} = R$.
\item[(ii)] $r(B_{k}) \sim r(B_{k + 1})$, and $|z - a| \sim r(B_{k})$ for $z \in CB_{k + 1} \setminus CB_{k}$. 
\item[(iii)] If $y \in 3B_{k}$ and $z \in \R^{n} \setminus CB_{k}$, then 
\begin{displaymath} |y - z| \sim |a - z|. \end{displaymath} 
\end{itemize}

The existence of such a sequence is a simple geometric fact. 
Then, using the sublinearity of $S$, one has
\begin{align*}  |S(f1_{\R^{n} \setminus CR})(a) - S(f1_{\R^{n} \setminus CB})(a)| & \leq \sum_{k = 0}^{K - 1} |S(f1_{\R^{n} \setminus CB_{k + 1}})(a) - S(f1_{\R^{n} \setminus CB_{k}})(a)|\\
& \leq \sum_{k = 0}^{K - 1} S(f1_{CB_{k + 1} \setminus CB_{k}})(a). \end{align*} 
So, in order to prove \eqref{rbmo2}, it suffices to obtain the following estimate for every individual term in the sum:
\begin{equation}\label{form2} S(f1_{CB_{k + 1} \setminus CB_{k}})(a) \lesssim \int_{CB_{k + 1} \setminus CB_{k}} \frac{d\mu(y)}{|y - a|^{m}}. \end{equation}
Write
\begin{displaymath} S^{2}(f1_{CB_{k + 1} \setminus CB_{k}})(a) = \int_{0}^{r(B_{k})} \cdots \frac{dt}{t^{m + 1}} + \int_{r(B_{k})}^{\infty} \cdots \frac{dt}{t^{m + 1}} =: I_{1} + I_{2}. \end{displaymath}
The pieces $I_{1}$ and $I_{2}$ both satisfy \eqref{form2}, as the following reasoning shows. In bounding $I_{1}$, the crucial fact is that if $t \leq r(B_{k})$, $y \in B(a,t) \subset 3B_{k}$ and $z \in CB_{k + 1} \setminus CB_{k}$, then $|y - z| \sim |a - z| \sim r(B_{k})$, combining the properties (ii) and (iii). This yields
\begin{align*} I_{1} & \lesssim \int_{0}^{r(B_{k})} \int_{B(a,t)} \left| \int_{CB_{k + 1} \setminus CB_{k}} \frac{t^{\alpha}}{|y - z|^{m + \alpha}} \, d\mu(z) \right|^{2} \, d\mu(y) \, \frac{dt}{t^{m + 1}}\\
& \lesssim \frac{[\mu(CB_{k + 1} \setminus CB_{k})]^{2}}{(r(B_{k}))^{2m}} \int_{0}^{r(B_{k})} \frac{t^{2\alpha - 1}}{(r(B_{k}))^{2\alpha}} \, dt \sim \left( \int_{CB_{k + 1} \setminus CB_{k}} \frac{d\mu(y)}{|y - a|^{m}} \right)^{2}.\end{align*} 
The estimate for $I_{2}$ requires even less care:
\begin{align*} I_{2} & \lesssim \int_{r(B_{k})}^{\infty} \int_{B(a,t)} \left| \int_{CB_{k + 1} \setminus CB_{k}} \frac{d\mu(z)}{t^{m}} \right|^{2} \, d\mu(y) \, \frac{dt}{t^{m + 1}}\\
& \lesssim \frac{[\mu(CB_{k + 1} \setminus CB_{k})]^{2}}{(r(B_{k}))^{2m}} \int_{r(B_{k})}^{\infty} \frac{(r(B_{k}))^{2m}}{t^{2m + 1}} \, dt \sim \left( \int_{CB_{k + 1} \setminus CB_{k}} \frac{d\mu(y)}{|y - a|^{m}} \right)^{2}. \end{align*} 
This completes the proof of \eqref{form2} -- and Theorem \ref{thm:rbmo}. \end{proof}

\subsection{$V$ can be bounded on $L^2(\mu)$ but unbounded on $L^p(\mu)$ for every $p>2$}

The purpose of this subsection is to demonstrate, by example, that the assumption of $V$ mapping $L^{2}(\mu) \to L^{2}(\mu)$ boundedly \textbf{does not} imply that $V$ maps $L^{p}(\mu)$ to $L^{p}(\mu)$ for any $p > 2$. In fact, our example shows that this implication fails even for $\mu = dx$, the Lebesgue measure on $\R$.

The heart of the example is the following lemma.
\begin{lem} There exists a measurable function $f \colon [0,1] \to (0,1]$ such that
\begin{equation}\label{L1} \int_{I} \ln^{+} \left( \frac{\ell(I)}{f(t)} \right) \, dt \lesssim \ell(I) \end{equation}
for all intervals $I \subset [0,1]$,
but
\begin{equation}\label{noLp} \int_{0}^{1} \left( \ln \frac{1}{f(t)} \right)^{p} \, dt = \infty, \qquad p > 1. \end{equation} 
\end{lem}

\begin{rem} Here $\ln^{+}$ is the non-negative logarithm; thus $\ln^{+} x = \ln \max\{1,x\}$. \end{rem}

\begin{proof} The function $f$ will be constructed explicitly as a product
\begin{displaymath} f = \prod_{n = 1}^{\infty} f_{n}. \end{displaymath}
Each factor $f_{n}$ has the form
\begin{displaymath} e^{-a_{n}}1_{E_{n}} + 1_{[0,1] \setminus E_{n}}, \end{displaymath}
where $E_{n} \subset [0,1]$ and $a_{n} \in \N$. Observe that for any $n \in \N$, one has
\begin{align*} \int_{0}^{1} \left(\ln \frac{1}{f(t)} \right)^{1 + 1/n} dt = \int_{0}^{1} \left(\sum_{k = 1}^{\infty} \ln \frac{1}{f_{k}(t)} \right)^{1 + 1/n} dt \geq \int_{E_{n}} a_{n}^{1 + 1/n} dt = |E_{n}| \cdot a_{n}^{1 + 1/n}. \end{align*}
The first requirement is that the sequence $(a_{n})_{n \in \N}$ tends to infinity so rapidly that
\begin{equation}\label{form7} |E_{n}| \cdot a_{n}^{1 + 1/n} \geq n. \end{equation}
This gives \eqref{noLp}, since for any $p > 1$ one has $\|\ln^{+} (1/f)\|_{L^{p}} \geq \|\ln^{+} (1/f)\|_{L^{1 + 1/n}}$ for large enough $n \in \N$.  On the other hand, taking $I = [0,1]$ shows that one should at least have
\begin{equation}\label{L1full} 1 \gtrsim \int_{0}^{1} \ln \frac{1}{f(t)} \, dt = \sum_{n = 1}^{\infty} \int_{0}^{1} \ln \frac{1}{f_{n}(t)} \, dt = \sum_{n = 1}^{\infty} \left( |E_{n}| \cdot a_{n} \right).\end{equation}
We can ensure this by requiring that, say,
\begin{displaymath} |E_{n}| \cdot a_{n} = 2^{-n}. \end{displaymath}
This condition and \eqref{form7} hold simultaneously, if one requires that $a_{n} \to \infty$ rapidly enough, and $|E_{n}| = 2^{-n}/a_{n}$.

The numbers $a_{n}$ have now been chosen, and their definition will no longer be tampered with. In order to verify \eqref{L1} for an arbitrary interval $I \subset [0,1]$, one only chooses the sets $E_{n}$ in an appropriate fashion. The construction is initialised by requiring that $E_{n}$ is a single interval of length $|E_{n}|$; during the process, each $E_{n}$ will be modified a finite number of times. If $I \subset [0,1]$ and $\ell(I) > e^{-a_{1}}$, one simply observes that $\ell(I) \gtrsim 1$, and \eqref{L1} follows from \eqref{L1full}. Next, suppose that the sets $E_{n}$ have already been chosen so that \eqref{L1} holds for all intervals $I \subset [0,1]$ of length $\ell(I) > \exp(-\sum_{n = 1}^{N} a_{n})$ for some $N \geq 1$. To proceed with the induction, one needs to consider the situation 
\begin{equation}\label{form8} \exp\Big(-\sum_{n = 1}^{N + 1} a_{n} \Big) < \ell(I) \leq \exp\Big(-\sum_{n = 1}^{N} a_{n}\Big). \end{equation}
The quantity on the right hand side being a lower bound for the product of the $N$ first functions $f_{n}$, one finds that 
\begin{displaymath} \prod_{n = 1}^{N} f_{n}(t) \geq \ell(I), \qquad x \in [0,1]. \end{displaymath}
With this in mind, and using $\max\{1,ab\} \leq b$ for $a \leq 1 \leq b$, one has
\begin{displaymath} \int_{I} \ln^{+} \left( \frac{\ell(I)}{f(t)} \right) \, dt \leq \sum_{n = N + 1}^{\infty} \int_{E_{n} \cap I} a_{n} \, dt. \end{displaymath}
To get the correct upper bound for this quantity, one only needs to choose the sets $E_{n}$, $n \geq N + 1$, so that the inequality
\begin{equation}\label{form9} |E_{n} \cap I| \leq 2 \cdot |E_{n}| \cdot \ell(I) \end{equation}
holds for all intervals $I \subset [0,1]$ satisfying the left hand side inequality of \eqref{form8}. A natural choice is to let $E_{n}$ be the union of $k = k_{N}$ roughly $1/k$-spaced subintervals of $[0,1]$, of length $|E_{n}|/k$. Then any interval $I \subset [0,1]$ intersects no more than $\ell(I)k + C$ of these subintervals, $C \geq 1$ being an absolute constant, so that
\begin{displaymath} |E_{n} \cap I| \leq (\ell(I)k + C) \cdot \frac{|E_{n}|}{k} = |E_{n}| \cdot \ell(I) + \frac{C|E_{n}|}{k}. \end{displaymath}
It remains to choose $k = k_{N}$ so large that $C/k \leq \exp(-\sum_{n = 1}^{N + 1} a_{n}) \leq \ell(I)$. 

Now \eqref{form9} yields
\begin{displaymath} \int_{I} \ln^{+} \left(\frac{\ell(I)}{f(t)} \right) \, dt \lesssim \ell(I) \cdot \sum_{n = N}^{\infty} \left( |E_{n}| \cdot a_{n} \right) \leq \ell(I), \end{displaymath}
and the proof is complete. \end{proof} 

With the lemma in use, it is easy to define a vertical square function operator $V$, which is bounded on $L^{2}(\R,dx)$, but such that $V(1_{[0,1]}) \notin L^{p}(\R,dx)$ for any $p > 2$. The kernel $s_{t}(x,y)$ is defined as follows. For $x \notin [0,1]$, set $s_{t}(x,\cdot) \equiv 0$ for all $t > 0$. In order to define $s_{t}(x,y)$ for $x \in [0,1]$, let $f$ be the function from the previous lemma, and set
\begin{displaymath} s_{t}(x,y) = \begin{cases} \varphi_{t}(x - y), & \text{if } f(x) \leq t \leq 1,\\ 0, & \text{otherwise.} \end{cases} \end{displaymath}
Here $\varphi$ is a smooth non-negative function with $1_{[-1,1]} \leq \varphi \leq 1_{[-2,2]}$, and $\varphi_{t}(z) = t^{-1}\varphi(z/t)$. Then, for $p > 2$,
\begin{align*} \|V(1_{[0,1]})\|_{L^{p}}^{p} & = \int_{0}^{1} \left(\int_{0}^{\infty} |\theta_{t}1_{[0,1]}(x)|^{2} \, \frac{dt}{t} \right)^{p/2} \, dx\\
& \geq \int_{0}^{1} \left( \int_{f(x)}^{1} \frac{dt}{t} \right)^{p/2} \, dx = \int_{0}^{1} \left( \ln \frac{1}{f(x)} \right)^{p/2} \, dx = \infty. \end{align*}
To prove that $V$ is bounded on $L^{2}$, it suffices (see for example \cite[Theorem 1.1]{MM}) to prove that 
\begin{displaymath} C_{I} := \int_{I} \int_{0}^{\ell(I)} |\theta_{t}1(x)|^{2} \, \frac{dt}{t} \, dx \lesssim \ell(I) \end{displaymath}
for all intervals $I \subset \R$. This is immediate from \eqref{L1}, the first condition on $f$. Since $|\theta_{t}1(x)| = 0$ for $x \notin [0,1]$ and $\theta_{t} \equiv 0$ for $t > 1$, one may assume that $I \subset [0,1]$. Then
\begin{displaymath} C_{I} \lesssim \int_{I} \int_{f(x)}^{\ell(I)} \frac{dt}{t} \, dx = \int_{I} \ln^{+} \left(\frac{\ell(I)}{f(x)} \right) \, dx \lesssim \ell(I), \end{displaymath}
as required.

\section{$S$ can be bounded on $L^2(\mu)$ even if $S_{\alpha}$, $\alpha > 1$, and $V$ are not}
In this section, we prove Theorem \ref{noTricks}. In other words, we construct a Borel probability measure $\mu$ and a square function operator $S$ on $\R$ with the following properties.
\begin{itemize}
\item[(i)] The measure $\mu$ and the kernel of the operator $S$ satisfy the assumptions \eqref{powerBound}, \eqref{eq:size} and \eqref{eq:hol} for some $0 < m < 1$. 
\item[(ii)] The operator $S$ is bounded on $L^{2}(\mu)$, but $S_{\alpha}(1) \notin L^{2}(\mu)$ for any $\alpha > 1$, where
\begin{displaymath} S_{\alpha}f(x) = \Big(\iint_{\Gamma_{\alpha}(x)} |\theta_{t}f(y)|^{2} \, \frac{d\mu(y)dt}{t^{m + 1}}\Big)^{1/2} \end{displaymath}
is the square function associated with the cones 
\begin{displaymath} \Gamma_{\alpha}(x) = \{(y,t) \in \R \times \R_{+} : |y - x| <  \alpha t\}. \end{displaymath}
\end{itemize}

The first concern is constructing the measure $\mu$. Fix $m \in (0,1/2)$. The construction of $\mu$ will proceed iteratively inside the interval $I_{0} = [0,1]$: given a generation $(n - 1)$ interval $I$, one chooses four disjoint compact subintervals $I_{1},I_{2},I_{3},I_{4} \subset I$ and describes how the $\mu$-mass of $I$ is divided among the intervals $I_{j}$. 

To get the construction started, define $I_{0}$ to be the only generation $0$ interval, and set $\mu(I_{0}) = 1$. Next suppose that the $\mu$-masses of all generation $(n - 1)$ intervals (denote this collection by $\mathcal{I}_{n - 1}$) have been determined for some $n \geq 1$ so that
\begin{equation}\label{exactMass} \mu(I) = \ell(I)^{m}, \qquad I \in \mathcal{I}_{n - 1}. \end{equation}
Fix $I \in \mathcal{I}_{n - 1}$. The lengths of the intervals $I_{1},I_{2},I_{3},I_{4} \subset I$ are determined by \eqref{exactMass} and the following requirements:
\begin{itemize}
\item[(i)] $\mu(I_{1}) + \mu(I_{2}) + \mu(I_{3}) + \mu(I_{4}) = \mu(I)$;
\item[(ii)] $\ell(I_{1}) = \ell(I_{4}) =: L_{I}$ and $\ell(I_{2}) = \ell(I_{3})$;
\item[(iii)] $\mu(I_{2}) = \mu(I_{3}) = \mu(I)/Cn$ and $\mu(I_{1}) = \mu(I_{4})$;
\item[(iv)] $\mu(I_{j}) = \ell(I_{j})^{m}$ for $1 \leq j \leq 4$.
\end{itemize}
In (iii), $C \geq 4$ is an absolute constant to be fixed in the course of the argument. These requirements -- as well as the placement of the intervals -- are depicted in Figure \ref{fig1}.
\begin{figure}[h!]
\begin{center}
\includegraphics[scale = 0.8]{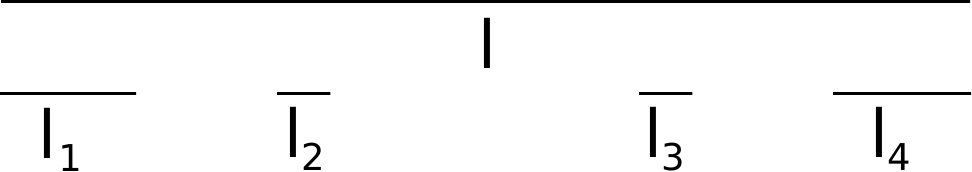}
\caption{The relation between generation $n$ and $(n + 1)$ intervals}\label{fig1}
\end{center}
\end{figure}
The placement of the intervals $I_{j}$ is described formally as follows. The intervals $I_{1}$ and $I_{4}$ have common boundary with $I$, whereas the the intervals $I_{2}$ and $I_{3}$ lie between $I_{1}$ and $I_{4}$. The final requirement is that
\begin{equation}\label{separation} \dist(I_{1},I_{2}) = L_{I} = \dist(I_{3},I_{4}). \end{equation}
The fact that $m < 1/2$ ensures that \eqref{separation} does not contradict \eqref{exactMass} or the conditions (i)-(iv). The definition of the measure $\mu$ is complete, so it is time to prove (i).



\begin{lem} For all intervals $J \subset \R$, one has $\mu(J) \lesssim \ell(J)^{m}$. \end{lem}
\begin{proof} Let $J \subset \R$ be an arbitrary interval. One may assume that $J \subset I_{0}$, since $\mu(J) = \mu(J \cap I_{0})$. The crucial observation is this: for any interval $I \in \mathcal{I}_{n}$, $n \in \N$,
\begin{equation}\label{separation2} \min\{\dist(I_{i},I_{j}) : 1 \leq i < j \leq 4\} \sim \ell(I). \end{equation}
This follows from \eqref{separation} and by taking $C \geq 4$ large enough in (iii). Thus, if $J$ intersects two of the intervals in $\mathcal{I}_{1}$, then $\ell(J) \gtrsim 1$, and the proof is complete. If not, then one may assume that $J \subset I_{j}$ for some $I_{j} \in \mathcal{I}_{1}$. Next, apply the same argument inside $I_{j}$: if $J$ intersects two of the second generation subintervals of $I_{j}$, one has
\begin{displaymath} \mu(J) \leq \mu(I_{j}) = \ell(I_{j})^{m} \lesssim \ell(J)^{m} \end{displaymath}
by \eqref{separation2}. Otherwise $\mu(I_{j} \cap J)$ is, once again, contained in a single second generation subinterval of $I_{j}$, and one may iterate the argument. Eventually, assuming that $\ell(J) > 0$, one encounters a situation where $J$, for the first time, intersects two distinct generation $(n + 1)$ subintervals inside a certain interval in $\mathcal{I}_{n}$. Then the reasoning above shows that $\mu(J) \lesssim \ell(J)^{m}$, and the proof is complete. \end{proof}

It remains to construct the operator $S$. To this end, one needs to fix the values of the kernel $s_{t}$, so let $x,y \in \R$ and $t > 0$ be arbitrary. Find out if there is $n \in \N$ and $I \in \mathcal{I}_{n}$ such that 
\begin{displaymath} x \in I_{2} \cup I_{3} \quad \text{ and } \quad \tfrac{L_{I}}{2} \leq t \leq L_{I}. \end{displaymath}
If this is not the case, set $s_{t}(x,y) = 0$. But if such an interval $I \in \mathcal{I}_{n}$ exists, let 
\begin{displaymath} s_{t}(x,y) = \frac{\varphi_{I}}{\ell(I)^{m}}, \end{displaymath} 
where $\varphi_{I}$ is a smooth bump function adapted to the interval $I$, with $0 \leq \varphi_{I} \leq 1_{I}$ and $\int \varphi \, d\mu \sim \mu(I) = \ell(I)^{m}$ (for example, one can take $\varphi_{I}(x) = \varphi([x - c_{I}]/\ell(I))$, where $c_{I}$ is the midpoint of $I$ and $\varphi$ is a non-negative smooth function with $\operatorname{spt} \varphi = [-1/2,1/2]$). Recalling that $t \sim L_{I} \sim \ell(I)$ for $L_{I}/2 \leq t \leq L_{I}$, it is easy to check that the size and smoothness conditions \eqref{eq:size} and \eqref{eq:hol} are satisfied.

Without further ado, we can now check that the square function operator associated with $\mu$ and the kernel $s_{t}$ is bounded on $L^{2}(\mu)$. Fix $f \in L^{2}(\mu)$ and write
\begin{align*} \|Sf\|_{L^{2}(\mu)}^{2} & = \iint_{\R \times \R_{+}} |\theta_{t}f(y)|^{2} \frac{\mu(B(y,t))}{t^{m}} \, d\mu(y) \, \frac{dt}{t}\\
& \leq \sum_{n = 0}^{\infty} \sum_{I \in \mathcal{I}_{n}} \int_{L_{I}/2}^{L_{I}} \int_{I_{2} \cup I_{3}} \left( \frac{1}{\ell(I)^{m}} \int_{I} |f(z)| \, d\mu(z) \right)^{2} \frac{\mu(B(y,t))}{t^{m}} \, d\mu(y) \, \frac{dt}{t}. \end{align*}
To estimate the latter expression further, we note that for $I \in \mathcal{I}_{n}$, $y \in I_{2} \cup I_{3}$ and $L_{I}/2 \leq t \leq L_{I}$, the separation condition \eqref{separation} implies that
\begin{displaymath} \mu(B(y,t)) = \mu(B(y,t) \cap [I_{2} \cup I_{3}]) \leq \mu(I_{2}) + \mu(I_{3}) = \frac{2\mu(I)}{C(n + 1)} \leq \frac{t^{m}}{n + 1}, \end{displaymath}
where the final estimate holds for $C \geq 4$ large enough. Combining this bound (applied twice) with the Cauchy-Schwarz inequality and the disjointness of the intervals in $\mathcal{I}_{n}$, we obtain
\begin{align*} \|Sf\|_{L^{2}(\mu)}^{2} & \leq \sum_{n = 0}^{\infty} \frac{1}{n + 1} \sum_{I \in \mathcal{I}_{n}} \left(\frac{1}{\ell(I)^{m}} \int_{I} |f(z)| \, d\mu(z) \right)^{2} \mu(I_{2} \cup I_{3})\\
& \leq \sum_{n = 0}^{\infty} \frac{1}{(n + 1)^{2}} \sum_{I \in \mathcal{I}_{n}} \frac{\mu(I)^{2}}{\ell(I)^{2m}} \int_{I} |f(z)|^{2} \, d\mu(z)\\
& = \sum_{n = 0}^{\infty} \frac{1}{(n + 1)^{2}} \|f\|_{L^{2}(\mu)}^{2} \sim \|f\|_{L^{2}(\mu)}^{2}. \end{align*}

All that remains is to verify that $S_{\alpha}(1) \notin L^{2}(\mu)$ for $\alpha > 1$, where $S_{\alpha}$ is the square function associated with the measure $\mu$, the kernel $s_{t}$ and the cones $\Gamma_{\alpha}$. First, recalling that $\int \varphi_{I} \, d\mu \sim \ell(I)^{m}$ for $I \in \mathcal{I}_{n}$, one has
\begin{equation}\label{S1} \|S_{\alpha}(1)\|_{L^{2}(\mu)}^{2} \sim \sum_{n = 0}^{\infty} \sum_{I \in \mathcal{I}_{n}} \int_{L_{I}/2}^{L_{I}} \int_{I_{2} \cup I_{3}} \frac{\mu(B(y,\alpha t))}{t^{m}} \, d\mu(y) \, \frac{dt}{t}. \end{equation}
So, one needs to study $\mu(B(y,\alpha t))$ for $y \in I_{2}$, say. Recall, once more, the separation condition \eqref{separation}. For $t = L_{I}$, the position of $I_{2}$ relative to $I_{1}$ was chosen precisely so that if $y_{l}$ is the left endpoint of $I_{2}$, then $\mu(B(y_{l},t))$ contains the right endpoint $y_{r}$ of $I_{1}$. Consequently, there exist constants $\tau_{\alpha} < 1$ and $c_{\alpha} > 0$ depending only on $\alpha > 1$ with the following property: if $\tau_{\alpha}L_{I} \leq t \leq L_{I}$, then
\begin{equation}\label{form4} B(y,\alpha t) \cap I_{1} \supset [y_{r} - c_{\alpha}L_{I}, y_{r}] \end{equation} 
for all $y \in [y_{l}, y_{l} + c_{\alpha}L_{I}]$. Since $\ell(I_{2}) < c_{\alpha}L_{I}$ for $n \geq n_{\alpha}$ (where $n$ is the generation of $I$), we actually see that, for such $n$, \eqref{form4} holds for all $y \in I_{2}$. The final observation is that
\begin{equation}\label{form5} \mu([y_{r} - c_{\alpha}L_{I}, y_{r}]) \gtrsim_{\alpha} \mu(I_{1}) = L_{I}^{m}, \end{equation}
which follows easily from the construction of $\mu$. Namely, the intersection $[y_{r} - c_{\alpha}L_{I},y_{r}] \cap I_{1}$ has to contain an interval of the form $I_{144\cdots4}$, where the number of iterations is bounded from above by a constant depending only on $c_{\alpha}$. Since each iteration decreases the $\mu$-measure of the interval by no more than a factor of $1/4$, we obtain \eqref{form5}.

Plugging the lower bound into \eqref{S1} yields
\begin{align}\label{form6} \|S_{\alpha}(1)\|_{L^{2}(\mu)}^{2} & \gtrsim_{\alpha} \sum_{n = n_{\alpha}}^{\infty} \sum_{I \in \mathcal{I}_{n}} \int_{\tau_{\alpha}L_{I}}^{L_{I}} \int_{I_{2}} \frac{L_{I}^{m}}{t^{m}} \, d\mu(y) \, \frac{dt}{t}\\
& \sim_{\alpha} \sum_{n = n_{\alpha}}^{\infty} \frac{1}{C(n + 1)} \sum_{I \in \mathcal{I}_{n}} \mu(I) = \sum_{n = n_{\alpha}}^{\infty} \frac{1}{C(n + 1)} = \infty. \notag \end{align}
This finishes the proof of Theorem \ref{noTricks}. To obtain Corollary \ref{noTricksCor}, observe that for any $\alpha > 1$ one has
\begin{align*} \infty = \|S_{\alpha}(1)\|_{L^{2}(\mu)}^{2} & = \iint_{\R^{2}_{+}} |\theta_{t}1(y)|^{2} \frac{\mu(B(y,\alpha t))}{t^{m}} \, d\mu(y) \, \frac{dt}{t}\\
& \lesssim_{\alpha} \iint_{\R^{2}_{+}} |\theta_{t}1(y)|^{2} \, d\mu(y) \, \frac{dt}{t} = \|V(1)\|_{L^{2}(\mu)}^{2}. \end{align*}
Thus $V(1) \notin L^{2}(\mu)$.


\section{Sharp weighted inequalities for square functions with doubling measures}
In this section we will prove Theorem \ref{weighted}. We assume that $\mu$ is doubling, has polynomial growth \eqref{powerBound} and that $S$ is of weak type $(1,1)$. Under these assumptions we show
the sharp weighted bounds for $S$. That is, we fix $p \in (1,\infty)$ and show that
$$ \|Sf\|_{L^p(w)} \lesssim_{n,p} [w]_{A_p}^{\max{(\frac{1}{2},\frac{1}{p-1}})} \|f\|_{L^p(w)},$$
where $$[w]_{A_p}=\sup_Q \left(\fint_Q w \,d\mu \right) \left(\fint_Q w^{-\frac{1}{p-1}} \,d\mu \right)^{p-1}.$$

First, a few definitions. Denote the average of a function $f$ over a cube or ball $Q$ by
$$\fint_Q f = \frac{1}{\mu(Q)} \int_Q f d\mu.$$
The non-increasing rearrangement of a $\mu-$measurable function $f$ on $\R^n$ is defined by
$$ f^*(t)= \inf \{\alpha>0: \mu({x\in \R^n : |f(x)|<\alpha}) <t \}, \qquad 0<t<\infty.$$
If $Q$ is a cube in $\R^n$, the local mean oscillation of $f$ on $Q$ is defined by
$$\osc{\lambda}(f;Q)= \inf_{c \in \R} \left((f-c)1_Q\right)^*\left(\lambda \mu(Q) \right).$$

Let us notice that obviously the exponent $\alpha>0$ in the pointwise and H\"{o}lder bounds for $s_t$ can be always considered to be strictly smaller than $1$. The following lemma is the key to the proof.
\begin{lem}
For any $Q\subset \R^n$,
\begin{equation}\label{weightedmainlemma} 
\osc{\lambda}\left( (Sf)^2 ; Q \right) \lesssim_{n,\lambda} \sum_{k\geq 0} 2^{-k \epsilon} \left( \fint_{2^k Q} |f| d\mu \right)^2,
\end{equation}
for some $\eps>0$.
\end{lem}
\begin{proof}
Let us fix a cube $Q \subset \R^n$ of side-length $\ell(Q)$ and a point $x \in Q$. Using the fact that $t^{-m} \lesssim {\mu(B(y,t))}^{-1}$, we see that
\begin{multline*}
(Sf)^2(x) \lesssim \int_0^{2 \ell(Q)}\!\!\! \int_{B(x,t)} |\theta_tf(y)|^2 \frac{d\mu(y)dt}{t^{m+1}}+ \sum_{k\geq 2} \int_{2^k \ell(Q)}^{2^{k+1} \ell(Q)} \!\!\! \int_{B(x,t)} |\theta_tf(y)|^2 \frac{d\mu(y)dt}{\mu(B(y,t))t}\\=:I_1f(x) + \sum_{k\geq 2} I_kf(x).
\end{multline*}
Therefore, it is enough to bound the right hand side of the inequality
\begin{equation}\label{oscillationbound}\osc{\lambda} ((Sf)^2;Q) \leq I_1^*f(x) + \sum_{k\geq 2} \|I_kf(x) - I_kf(c_Q)\|_{L^\infty(Q)},\end{equation}
where $c_Q$ is the center of $Q$. If $R$ is any cube in $\R^n$ and $(y,t) \in R \times (0,\ell(R))$ then by \eqref{eq:size} and $|y-z| \sim |z-c_R| \sim 2^j \ell(R)$
\begin{equation*}
|\theta_t (f 1_{2^{j+1}R\setminus 2^j R})(y)| \lesssim t^\alpha \int_{2^{j+1}R\setminus 2^j R} \frac{|f(z)|d\mu(z)}{|z-c_R|^\alpha \mu\left(B\left(c_R, 2^j \ell(R)\right)\right)}\end{equation*}
\begin{equation}
\leq \left(\frac{t}{2^j \ell(R)}\right)^\alpha \frac{1}{\mu\left(B\left(c_R, 2^j \ell(R)\right)\right)}\int_{2^{j+1}R\setminus 2^j R}|f(z)|d\mu(z)=:A(j,R),
\label{weightedannular}\end{equation}
for $j \geq 2$. Moreover, if $(y,t) \in R \times (\ell(R)/2,\ell(R))$ by \eqref{eq:size} and the doubling property of $\mu$ we have
\begin{equation}\label{weightedlocal}
|\theta_t (f 1_{4R})(x)|\lesssim \int_{4R} \frac{|f(z)|d\mu(z)}{\mu\left(B\left(y,t\right)\right)} \lesssim  \frac{1}{\mu\left(B\left(c_R, 4 \ell(R)\right)\right)}\int_{4R} |f(z)|d\mu(z)=:A(1,R).
\end{equation}
We introduce now a Lipschitz  cut-off function $\phi$ such that $1_{B(0,1)} \leq \phi \leq 1_{B(0,2)}$. Therefore, by $|x-y| \sim |y-c_Q| < t \sim 2^k \ell(Q)$, which also entails that $y \in 2^{k+2} Q$, we have
\begin{multline*}
|I_k f(x) - I_k f(c_Q)| \leq\\ 
\int_{2^k \ell(Q)}^{2^{k+1}\ell(Q)} \int_{2^{k+2}Q} | \phi \left(|x-y|/t \right) - \phi \left(|c_Q-y|/t\right) | |\theta_t f(y)|^2 \frac{d\mu(y)dt}{\mu(B(y,t))t}\\
\leq \|\phi\|_{Lip}  \int_{2^k \ell(Q)}^{2^{k+1}\ell(Q)} \int_{2^{k+2}Q}\frac{|x-c_Q|}{t} |\theta_t f(y)|^2 \frac{d\mu(y)dt}{\mu(B(y,t))t}\\
\lesssim 2^{-k}  \int_{2^k \ell(Q)}^{2^{k+1}\ell(Q)}\!\!\!\! \int_{2^{k+2}Q}|\theta_t (f 1_{4 \cdot 2^{k+2}Q})(y)|^2  \frac{d\mu(y)dt}{\mu(B(y,t))t} \\
+ 2^{-k}  \int_{2^k \ell(Q)}^{2^{k+1}\ell(Q)}\!\!\!\! \int_{2^{k+2}Q} \sum_{j\geq 2} |\theta_t (f 1_{2^{j+1} 2^{k+2}Q \setminus 2^j 2^{k+2}Q})|^2  \frac{d\mu(y)dt}{\mu(B(y,t))t},
\end{multline*}
which by \eqref{weightedannular} and \eqref{weightedlocal} for $R=2^{k+2}Q$ and $\mu(B(y,t)) \sim \mu(2^{k+2}Q)$ is bounded by a constant multiple of
\begin{multline*} 2^{-k}  \int_{2^k \ell(Q)}^{2^{k+1}\ell(Q)}\!\!\!\! \int_{2^{k+2}Q}
\left(A(1,2^{k+2}Q)^2 + \sum_{j\geq 2} A(j,2^{k+2}Q)^2\right) \frac{d\mu(y)dt}{\mu(B(y,t))t}\\
\lesssim 2^{-k}  \fint_{2^k \ell(Q)}^{2^{k+1}\ell(Q)}\!\!\!\! \fint_{2^{k+2}Q} \left(  \fint_{2^{k+4} Q} |f(z)| d\mu(z) \right)^2 d\mu(y)dt \\ + 2^{-k}  \fint_{2^k \ell(Q)}^{2^{k+1}\ell(Q)}\!\!\!\! \fint_{2^{k+2}Q} \sum_{j\geq 2}\left( \left(\frac{2^k \ell(Q)}{2^{j+k} \ell(Q)}\right)^\alpha \fint_{2^{j+k+3}Q}|f(z)|d\mu(z)\right)^2 d\mu(y)dt\\
\lesssim 2^{-k} \left( \fint_{2^{k+4} Q} |f(z)| d\mu(z) \right)^2 + 2^{-k} \left( \sum_{j\geq 2}2^{-j\alpha}\fint_{2^{j+k+3}Q}|f(z)|d\mu(z)\right)^2,\end{multline*}
where in the last step we used H\"{o}lder's inequality. Notice that the first term is exactly what we were after while the sum of the second term in $k \geq 2$ after we substitute $i=k+j$ and apply H\"{o}lder's inequality is bounded by
$$\sum_{k \geq 2} 2^{-k(1-\alpha)} \sum_{i\geq k+2} 2^{-i \alpha} \left(\fint_{2^{i+3}Q}|f(z)|d\mu(z)\right)^2 \lesssim  \sum_{i\geq 4}2^{-i \alpha} \left(\fint_{2^{i+3}Q}|f(z)|d\mu(z)\right)^2.$$
It remains to show that $$(I_1(f)1_Q)^*(\lambda \mu(Q)) \lesssim  \sum_{k\geq 0} 2^{-k \epsilon} \left( \fint_{2^k Q} |f| d\mu \right)^2.$$
To do so, we write $f= f 1_{4Q} + (f 1_{\R^n \setminus 4Q})$ and therefore, by the sublinearity of our operator,
$$(I_1(f)1_Q)^*(\lambda \mu(Q)) \lesssim (I_1(f 1_{4Q})1_Q)^*(\lambda \mu(Q))+(I_1(f 1_{\R^n \setminus 4Q})1_Q)^*(\lambda \mu(Q)),$$
which, by the weak-type $(1,1)$ boundedness of $S$ for the first term and Chebyshev's inequality, Fubini, \eqref{weightedannular} and H\"{o}lder's inequality for the second, is $\lesssim$
$$\left( \fint_{4Q}|f|d\mu\right)^2 + \fint_Q \int_0^{2 \ell(Q)} \int_{B(x,t)} |\theta_t (f 1_{\R^n \setminus 4Q})|^2 \frac{d\mu(y)dt}{\mu(B(y,t))t} d\mu(x)$$
$$\lesssim \left( \fint_{4Q}|f|d\mu\right)^2 + \fint_{3Q}  \int_0^{2 \ell(Q)} |\sum_{k \geq 2} \theta_t (f 1_{2^{k+1}Q \setminus 2^{k}Q})|^2 \frac{d\mu(y)dt}{t}$$
$$\lesssim \left( \fint_{4Q}|f|d\mu\right)^2 + \int_0^{2 \ell(Q)} \left(\sum_{k \geq 2} \left(\frac{t}{2^k \ell(Q)}\right)^\alpha \frac{1}{\mu\left(2^{k+1} Q\right)}\int_{2^{k+1}Q}|f(z)|d\mu(z)\right)^2\frac{dt}{t}$$
$$\lesssim \left( \fint_{4Q}|f|d\mu\right)^2 + \sum_{k \geq 2}  2^{-k \alpha} \left( \fint_{2^{k+1} Q} |f| d\mu \right)^2,$$
which concludes the proof of the lemma.
\end{proof}

\subsection*{Proof of the weighted bound}
After the previous lemma, the weighted bound follows by inspection of the proof of Theorem 1.1 of \cite{Le2}.

\end{document}